\documentclass[a4paper]{amsart}
\usepackage{amsmath,amsthm,amssymb, amscd, amsfonts, soul}

\usepackage[dvipsnames]{xcolor}

\theoremstyle{plain}
\newtheorem{theorem}{Theorem}[section]
\newtheorem{corollary}[theorem]{Corollary}

\newtheorem{proposition}[theorem]{Proposition}
\newtheorem{lemma}[theorem]{Lemma}
\theoremstyle{definition}
\newtheorem{definition}[theorem]{Definition}

\theoremstyle{remark}
\newtheorem{remark}[theorem]{Remark}

\makeatletter              % This sequence of commands will
\let\c@equation\c@theorem  % incorporate equation numbering
\makeatother

\numberwithin{equation}{section}\theoremstyle{plain}

\newcommand{\C}{{\mathcal C}}

\newcommand{\Z}{{\mathbb Z}}
\newcommand{\D}{{\mathcal D}}
\newcommand{\G}{{\mathbb G}}

\newcommand{\M}{{\mathcal M}}

\newcommand{\F}{{\mathbb F}}
\newcommand{\kk}{{\Bbbk}}

\newcommand\lRep{\text{-}\operatorname{mod}}
\newcommand\id{\operatorname{id}}

\newcommand\Aut{\operatorname{Aut}}
\newcommand\co{\operatorname{co}}

\newcommand\UT{\operatorname{UT}}

\newcommand\Opext{\operatorname{Opext}}

\newcommand\comod{\operatorname{-comod}}

\newcommand\Rep{\operatorname{mod-}}
\newcommand\Fun{\operatorname{Fun}}

\newcommand\vect{\operatorname{Vec}}

\newcommand{\ldual}{\underline{\lhd}}
\newcommand{\rdual}{\underline{\rhd}}

\begin{document}
\title[Deformations of semisimple Hopf algebras of dimension $p^3$, $pq^2$]{Cocycle deformations and Galois objects  for semisimple Hopf algebras of dimension $p^3$ and $pq^2$}

\author[A. Mej\'{i}a Casta\~{n}o]{Adriana Mej\'{i}a Casta\~{n}o }
\address{Mej\'{i}a Casta\~{n}o: Departamento de Matem\'{a}tica, Universidade Federal de Santa Catarina, CAPES - PNPD, Florian\'{o}polis-SC, Brazil.}
\email{sighana25@gmail.com}

\author[S. Montgomery]{Susan Montgomery}
\address{Montgomery: Department of Mathematics, University of Southern California, Los Angeles, California 90089, USA}
\email{smontgom@usc.edu}

\author[S. Natale]{Sonia Natale}
\address{Natale: Facultad de Matem\'{a}tica, Astronom\' \i a, F\' \i sica y Computaci\' on, Universidad Nacional de C\'{o}rdoba, CIEM-CONICET, C\'{o}rdoba, Argentina.}
\email{natale@famaf.unc.edu.ar}

\author[M. Vega]{\\Maria D. Vega}
\address{Vega: Department of Mathematical Sciences, United States Military Academy, West Point, New York 10996, USA.}
\email{maria.vega@usma.edu}

\author[C. Walton]{Chelsea Walton}
\address{Walton: Department of Mathematics, Temple University, Philadelphia, Pennsylvania 19122, USA}
\email{notlaw@temple.edu}

\begin{abstract}  Let $p$ and $q$ be distinct prime numbers. We study the Galois objects and cocycle deformations of the noncommutative, noncocommutative, semisimple Hopf algebras of odd dimension $p^3$ and of dimension $pq^2$. We obtain that the $p+1$ non-isomorphic self-dual semisimple Hopf algebras of dimension $p^3$ classified by Masuoka have no non-trivial cocycle deformations, extending his previous results for the 8-dimensional Kac-Paljutkin Hopf algebra. This is done as a consequence of the classification of categorical Morita equivalence classes among  semisimple  Hopf algebras of odd dimension $p^3$,  established by the third-named author in an appendix.
\end{abstract}



\maketitle

\section{Introduction}

Throughout this paper we shall work over an algebraically closed field $\kk$ of characteristic zero, and let $p$ and $q$ be distinct prime numbers.

\medbreak Within the study of semisimple Hopf algebras and fusion categories, we are motivated  by the vital interplay between {\sf bi-Galois objects}, {\sf cocycle deformations}, {\sf monoidal Morita-Takeuchi equivalence}, and  {\sf categorical  Morita equivalence}. These notions will be recalled below; the interplay is described in Proposition~\ref{prop-connect}. Our achievement here is that we obtain results pertaining to these notions for the (comodule categories of the) semisimple Hopf algebras of dimension $p^3$ classified by Masuoka \cite{masuoka-pp}. Namely, via a systematic approach using group-theoretical fusion categories, we parameterize the isomorphism classes of right Galois objects for the nontrivial semisimple Hopf algebras of dimension $p^3$. We also determine both the cocycle deformations and categorical Morita equivalence classes of these Hopf algebras. We achieve similar results on Galois objects and cocycle deformations for  the semisimple  Hopf algebras of dimension $pq^2$ classified by the third author \cite{natale-pqq}. See Theorems~\ref{mainintro}, \ref{mainintro-morita}, and~\ref{mainintro-pqq}, respectively.

\medbreak To begin, take $H$ and $L$ arbitrary Hopf algebras with  standard structure notation $m, u, \Delta, \epsilon, S$. We employ Sweedler notation (e.g. $\Delta(h) =  \sum h_{(1)} \otimes h_{(2)}$, for $h \in H$).

\medbreak \noindent {\sf (Bi)Galois objects.} ~~A nonzero algebra $R$ is a {\it right $H$-Galois object} if $R$ is an
$H$-comodule algebra such that $R^{\co H} \cong \kk$ and the
Hopf-Galois map
$R\otimes R \overset{\sim}{\to} R \otimes H$, $r \otimes s \mapsto \sum rs_{(0)} \otimes s_{(1)}$
is bijective; a {\it left $L$-Galois object} $R$ is defined analogously. An {\it $(L,H)$-biGalois object} $R$ is a nonzero right $H$- Galois object  and a left $L$- Galois object  for which $H$ and $L$ coact on $R$ compatibly.
If $R$ is a right $H$-Galois object, then there exists a Hopf algebra  $L(R, H)$   that coacts on $R$ from the left so that $R$ is an $(L(R, H),H)$-biGalois object. Furthermore, $L(R, H)$ is called the {\it left Galois Hopf algebra} of $R$, and it is unique up to isomorphism.  A right $H$-Galois object $R$ is  {\it trivial} if $R = H$ as right $H$-comodule algebras; in this case,  $L(R, H) = H$. See  \cite{sch} for more details.

\medbreak \noindent {\sf Cocycle deformations.} ~~  Let $\sigma : H \otimes H \to \kk$ be a 2-cocycle, that is, $\sigma$ is invertible under the convolution product, and
$$\textstyle \sum \sigma(x_{(1)}, y_{(1)}) \sigma(x_{(2)}y_{(2)}, z) = \sum \sigma(y_{(1)},z_{(1)}) \sigma(x,y_{(2)}z_{(2)}), \; \sigma(x,1) = \sigma(1,x) = \epsilon(x),$$
for $x,y,z \in H$. The {\it (2-)cocycle deformation} $H^{\sigma}$ of $H$ is  equal to $H$ as $\kk$-coalgebras, with multiplication and antipode given by
\begin{align*}& x \ast_\sigma y = \textstyle \sum \sigma(x_{(1)}, y_{(1)}) x_{(2)} y_{(2)} \sigma^{-1}(x_{(3)},y_{(3)}), \\
& S_{\sigma}(x) = \textstyle \sum \sigma(x_{(1)}, S(x_{(2)}))S(x_{(3)})\sigma^{-1}(S(x_{(4)}),x_{(5)}), \end{align*}
for $x,y \in H$. If $H^\sigma \cong H$, as Hopf algebras, then  the 2-cocycle deformation $H^\sigma$ of $H$ is called {\it trivial}. See \cite{doi} for further background.

\medbreak \noindent {\sf Monoidal Morita-Takeuchi equivalence.} ~~  We say that  Hopf algebras $H$ and $L$ are {\it monoidally Morita-Takeuchi equivalent} if the categories of  left  finite-dimensional  comodules of $H$ and  $L$, denoted $H\comod$ and $L\comod$, respectively, are equivalent as  $\kk$-linear   monoidal categories.\footnote{ This notion is sometimes called {\it monoidal co-Morita equivalence} in the literature.  }  We refer the reader to \cite{sch, sch2} for more details.

\medbreak \noindent {\sf Categorical Morita equivalence.} ~~ Let $\mathcal{C}$ and $\mathcal{D}$ be  tensor  categories, and for an  exact  indecomposable $\mathcal{C}$-module category $\mathcal{M}$, let  $\mathcal{C}_{\mathcal{M}}^*$ denote the tensor category of $\mathcal{C}$-module category endofunctors  of $\mathcal{M}$. We say that $\mathcal{C}$ and $\mathcal{D}$ are {\it categorically Morita equivalent}  if $\mathcal{C}_{\mathcal{M}}^*$ and $\mathcal{D}^{op}$ are equivalent as tensor categories,  where $\D^{op} = \D$, with reversed tensor product. See \cite[Section~7.12]{EGNO} for more information.
Further, if $\mathcal{C} = H\comod$ and $\mathcal{D} = L\comod$ are the categories of finite-dimensional comodules of $H$ and $L$, respectively, then we say that $H$ and $L$ are {\it categorically Morita equivalent} if this condition holds for $\mathcal{C}$ and $\mathcal{D}$.

\medbreak Let us now recall some ties between the  notions   above.

\begin{proposition}\label{prop-connect} (Bi/right) Galois objects, cocycle deformations, monoidal Morita-Takeuchi equivalence, and categorical Morita equivalence,  are connected in the following ways. Let $H$ and $L$ be finite-dimensional Hopf algebras.

\begin{enumerate}
\item  \cite[Theorems~9 and 11(3)]{doi-tak}  Any right $H$-Galois object is  a {\it crossed product algebra}
${}_\sigma H:=\kk \#_\sigma H$  with $x \ast_\sigma y = \textstyle  \sum \sigma(x_{(1)}, y_{(1)}) x_{(2)} y_{(2)}$
 for $x,y \in H$ and $\sigma$ a convolution invertible 2-cocycle of $H$.
\smallskip

\item \cite[Theorem~3.9]{sch} The crossed product algebra ${}_\sigma H$ is an $(H^\sigma, H)$-biGalois object where  the coaction ${}_\sigma H \to H^\sigma \otimes {}_\sigma H$ is given by $x \mapsto \sum x_{(1)} \otimes x_{(2)}$, and the left Galois Hopf algebra of ${}_\sigma H$ is $H^\sigma$.
\smallskip

\item    \cite[Theorem~5.5 and Corollary~5.7]{sch}  Up to isomorphism,  every monoidal Morita-Takeuchi equivalence  $\Phi : H\textnormal{-comod} \to L\textnormal{-comod}$ is given by taking a co-tensor product with a unique $(L,H)$-biGalois object $R$  which is unique up to isomorphism. Namely, $\Phi  \cong  \Phi_R$ where $\Phi_R(V) =  R \square_H V$, for all $V \in H$\textnormal{-comod}.
\smallskip

\item  \cite[Corollary~5.9]{sch} The Hopf algebras $H$ and $L$ are monoidally Morita-Takeuchi equivalent if and only if $L$ is a cocycle deformation of $H$.
\smallskip

\item  If the equivalent conditions of  (4) hold, then $H$ and $L$ are categorically Morita equivalent.   \qed
\end{enumerate}
\end{proposition}

 \begin{remark}Part (1) of Proposition 1.1 can also be deduced from \cite{KT} and \cite{SM} as follows.
If $R^{\co H} = \kk \subseteq R$ is $H$-Galois, then it is finite-dimensional by \cite[Theorem~1.7(1)]{KT}. By the Krull-Schmidt Theorem, the $H^*$-module $R$ is isomorphic to $H$ as $H$-comodules, and thus is $H$-cleft. Now  \cite[Theorem~7.2.2]{SM} implies that $R$ is isomorphic to a crossed product $\kk \#_\sigma H$.
\end{remark}

We examine two classes of  noncommutative, noncocommutative, semisimple Hopf algebras. The first class we study are Masuoka's Hopf algebras
$$A_{\zeta, 1}, ~A_{\zeta^t, 1}, ~A_{\zeta, g}, ~A_{\zeta^2, g}, \dots, ~A_{\zeta^{p-1}, g}$$ of dimension $p^3$, where $p$ is an odd prime number \cite[Theorem~3.1]{masuoka-pp}; we refer to reader to Definition~\ref{def-A} for their presentation. In particular, $g$ is a group-like element  in a certain commutative Hopf algebra,  $\zeta$ is a fixed $p$th root of unity, and  $t \in \mathbb{F}_p$ is a quadratic nonresidue.

\medbreak 
The main results of this paper concerning cocycle deformations and Galois objects for these Hopf algebras are summarized in the following theorem.

\begin{theorem}[Theorem~\ref{mainGalois}, Proposition~\ref{cocycle-ppp}] \label{mainintro} Consider Masuoka's  semisimple Hopf algebras of dimension $p^3$ listed above.

\smallskip
\begin{enumerate}
\item If $p = 3$, then the following statements hold.

\smallskip

\begin{itemize}
\item[(a)] Each right Galois object for the Hopf algebras $A_{\zeta, 1}$ and $A_{\zeta^2, 1}$ is trivial.

\item[(b)] The Hopf algebras $A_{\zeta, g}$ and $A_{\zeta^{2}, g}$ both have exactly two right Galois objects up to isomorphism.

\end{itemize}
\smallskip
 \item If $p > 3$, then the following statements hold.

 \smallskip

 \begin{itemize}
 \item[(a)] Each right Galois object for the Hopf algebras $A_{\zeta, g}, \dots, A_{\zeta^{p-1}, g}$ is trivial.

 \item[(b)] The Hopf algebras $A_{\zeta, 1}$ and $A_{\zeta^t, 1}$  both have exactly $p$ right Galois objects up to isomorphism.
\end{itemize}

\smallskip

\item Further, none of these Hopf algebras admits a non-trivial
cocycle deformation.\end{enumerate}
\end{theorem}

Part (3) of Theorem \ref{mainintro} is obtained as a consequence of the following theorem, established by the third author in Appendix \ref{morita-ppp}.

\medbreak
Let us denote by $G = \UT(3, p)$ the group of upper triangular unipotent $3\times 3$ matrices with entries in the field $\F_p$, and by $T = \Z_{p^2} \rtimes \Z_p$  the unique nonabelian group of order $p^3$ and exponent $p^2$.

\begin{theorem}[Theorem~\ref{cls-morita}]\label{mainintro-morita} Semisimple Hopf algebras of dimension $p^3$ fall into $p+6$ categorical Morita equivalence classes. More precisely, the Hopf algebras
\begin{equation*}\kk^{\Z_p \times \Z_p \times \Z_p}, \; \kk^{\Z_p \times \Z_{p^2}}, \; \kk^{\Z_{p^3}},  \kk^G, \; \kk^T, \; A_{\zeta, 1}, \; A_{\zeta^t, 1}, \; A_{\zeta, g}, \; \dots, \; A_{\zeta^{p-1}, g},
\end{equation*} are pairwise categorically Morita inequivalent and, furthermore,
the equivalence class of $\kk^G$ (respectively, the equivalence class of $\kk^T$) consists of $\kk^G$ and $\kk G$ (respectively, of $\kk^T$ and $\kk T$).
\end{theorem}

\medbreak

The second class of noncommutative, noncocommutative, semisimple Hopf algebras we study are those of dimension $pq^2$.
According to work of the third author \cite[Theorem 3.12.4]{natale-pqq}, these Hopf algebras are divided into three classes:
$$A_l, \text{ for } p \equiv 1~\text{mod } q, \quad \quad  B_{\lambda}, B_{\lambda}^*, \text{ for } q \equiv 1~\text{mod } p.$$
 See Definitions~\ref{A-l} and~\ref{B-lambda} for their presentations.
Here, $l$ is an integer between 0 and $q-1$,  and $\lambda$ runs over a certain set of integers between 0 and  $p-2$.
Our main result here is the following.

\begin{theorem}[Theorem~\ref{main-teo}, Proposition~\ref{cocycle-pqq}] \label{mainintro-pqq}
Consider  the semisimple Hopf algebras of dimension $pq^2$, $A_l$, $B_\lambda$, listed above.

\smallskip
\begin{enumerate}
\item Each right Galois object for the Hopf algebras $B_{\lambda}$ and $A_l$, for $l \neq 0$, is trivial;
 these Hopf algebras  do not admit   any non-trivial cocycle deformation.

\smallskip
\item The Hopf algebra $A_0$ has exactly $q$ right Galois objects up to isomorphism, and it  is a cocycle deformation of a commutative Hopf algebra.
\smallskip

\item  Each of the Hopf algebras $B_{\lambda}^*$  has exactly $\frac{p+q-1}{p}$ right Galois objects up to
isomorphism, and is a cocycle deformation of a commutative Hopf algebra.
\end{enumerate}
\end{theorem}

We achieve the results on Galois objects above by first realizing each Hopf algebra as an abelian extension corresponding to a matched pair of finite groups.
Then we use a bijective correspondence between right Galois objects for abelian extensions and fiber functors of  certain group-theoretical fusion categories associated to the relevant matched pair  (Proposition~\ref{bijcorr}). Finally, we
employ Ostrik's and the third author's parameterization of fiber functors of group-theoretical fusion categories  (Theorem~\ref{th-ostrik}; see also Remark~\ref{rmk-ppp-pqq}).

\medbreak Results  related to those in  Theorem~\ref{mainintro} were obtained by Masuoka for certain semisimple Hopf algebras of even dimension \cite{ma-contemp}.  In particular, he showed  that the Kac-Paljutkin Hopf algebra of dimension 8 has only trivial Galois objects and trivial cocycle deformations \cite[Theorems~4.1 and~4.8]{ma-contemp}. We recover this result in  Appendix \ref{kp}  (Proposition~\ref{prop-H8}) via  the techniques specified above.

\medbreak The paper is organized as follows. In Section~\ref{grpthl}, we provide background material on group-theoretical fusion categories, their module categories, and discuss results of Ostrik and of the third author crucial to the proof of Theorems~\ref{mainintro} and~\ref{mainintro-pqq}. We then provide background material on abelian extensions of Hopf algebras in Section~\ref{ab-ext}, and  recall  the connection between Galois objects and fiber functors  in Section~\ref{sec-Gal}. We review Masuoka's classification of semisimple Hopf algebras of dimension $p^3$ in Section~\ref{s-ppp}.   Parts (1) and (2) of Theorem~\ref{mainintro} are established in Section~\ref{sec-main1}.   See Section~\ref{s-pqq} and~\ref{sec-main2} for a discussion of the semisimple Hopf algebras of dimension $pq^2$ and the  proof   of Theorem~\ref{mainintro-pqq}. Results on  cocycle deformations  are presented in Section~\ref{s-cocycle}.  These methods are applied in Appendix \ref{kp}  to yield  Masuoka's  results on the 8-dimensional Kac-Paljutkin Hopf algebra. The classification of categorical Morita equivalence classes among the noncommutative noncocommutative examples in dimension $p^3$ is given in Appendix  ~\ref{morita-ppp}.

%%%%%%%%%%%%%%%%%%%%%%%%%%
%%%%%%%%%%%%%%%%%%%%%%%%%%
%%%%%%%%%%%%%%%%%%%%%%%%%%
%%%%%%%%%%%%%%%%%%%%%%%%%%

\section{Group-theoretical fusion categories} \label{grpthl}

In this section, we discuss background material on group-theoretical fusion categories and their module categories, ending with results of Ostrik \cite{ostrik}  and the third author \cite{pmc} that are crucial for establishing our main result.

\subsection{Fusion categories and their module categories}\label{modcat}

We refer to \cite{EGNO, ENO} for a general theory of such categories.
Let $\mathcal{C}$ be a fusion category over $\kk$. A \emph{left module category} over $\mathcal{C}$ (or $\C$-\emph{module category}) is a finite semisimple $\kk$-linear abelian category $\mathcal{M}$ equipped with an action bifunctor $\otimes:\mathcal{C}\times\mathcal{M}\rightarrow\mathcal{M}$ and natural isomorphisms
$$ m_{X,Y,M}:(X\otimes Y)\otimes M\rightarrow X\otimes(Y\otimes M),\hspace{1.5cm} u_M:\textbf{1}\otimes M\rightarrow M, $$
for $X, Y \in \C$, $M \in \mathcal M$, satisfying appropriate coherence conditions \cite[Section~7.1]{EGNO}. A module category structure on $\M$ corresponds to a tensor functor \linebreak $F:\C \to \Fun(\M, \M)$, such that $F(X)(M) = X \otimes M$.
The \emph{rank} of $\M$ is the cardinality of the set of isomorphism classes of simple objects of $\M$. Module categories of rank one correspond to tensor functors $\C \to \Fun(\M, \M) \cong \vect$, that is, to {\it fiber functors}  on $\C$.

\medbreak A $\mathcal{C}$-module category $\M$ is called \emph{indecomposable} if it is not equivalent to a direct sum of two non-trivial $\mathcal{C}$-submodule categories. Every semisimple indecomposable module category over $\C$ is equivalent to the category $\C_A$ of right $A$-modules in $\C$, for some semisimple indecomposable algebra $A$ in $\C$ \cite[Section~3.3]{ostrik}.

\medbreak Let $\M$ be an indecomposable $\C$-module category. The category $\mathcal{C}^*_{\mathcal{M}}$ of $\mathcal{C}$-module endofunctors of $\mathcal{M}$ is a fusion category.
A fusion category $\mathcal{D}$ is called \emph{categorically Morita equivalent} to $\mathcal{C}$  if there exists an indecomposable $\mathcal{C}$-module category $\mathcal{M}$ and an equivalence of tensor categories $\mathcal{D}\cong (\mathcal{C}^*_{\mathcal{M}})^{\text{op}}$.  See \cite[Section 7.12]{EGNO} for details.

\medbreak Suppose that $\M \cong \C_A$ for some semisimple indecomposable algebra $A$ in $\C$. Then there is an equivalence of fusion categories $(\mathcal{C}^*_{\mathcal{M}})^{op} \cong {}_A \C_A$, where ${}_A\C_A$ is the category of $A$-bimodules in $\C$; the latter is a fusion category with tensor product $\otimes_A$ and unit object~$A$. See \cite[Remark 4.2]{ostrik} or
\cite[Remark~7.12.5]{EGNO}.

\medbreak  Two fusion categories are categorically Morita equivalent if and only if their Drinfeld centers are equivalent as braided fusion categories \cite[Theorem 3.1]{ENO2}.

\subsection{Group-theoretical fusion categories}\label{gen-gt} We refer the reader to \cite[Section~9.7]{EGNO} and \cite[Section~8.8]{ENO} for general background material. A fusion category $\C$ is called \emph{pointed} if every simple object of $\C$ is invertible.
If $\C$ is a pointed fusion category, then there exist a finite group $G$ and a 3-cocycle $\omega:G \times G \times G \to \kk^{\times}$ such that $\C$ is equivalent to the fusion category  $\C(G, \omega)$  of finite-dimensional $G$-graded vector spaces with associativity constraint determined by $\omega$.

\medbreak Every indecomposable module category over the fusion category $\C(G, \omega)$ arises from a pair $(F, \alpha)$, where $F$ is a subgroup of $G$ such that the class of $\omega\vert_{F\times F \times F}$ is trivial in $H^3(F, \kk^\times)$ and $\alpha: F \times F \to \kk^{\times}$ is a 2-cochain on $F$ satisfying $d\alpha=\omega|_{F\times F\times F}$; see, e.g., \cite[Example~2.1]{ostrik-dd}.
If $(F, \alpha)$ is such a pair, then the twisted group algebra $\kk_\alpha F$ is a semisimple, indecomposable, associative algebra in $\C(G, \omega)$, and the module category corresponding to the pair $(F, \alpha)$ is the category $\M_0(F, \alpha) := \C(G, \omega)_{\kk_\alpha F}$ of right  $\kk_\alpha F$-modules in $\C(G, \omega)$. We shall use the notation
$$\C(G, \omega, F, \alpha):= \C(G, \omega)^*_{\M_0(F, \alpha)}.$$

\medbreak A fusion category $\C$ is called {\it group-theoretical} if it is categorically Morita equivalent to a pointed fusion category.  An indecomposable module category $\M$ such that $\C^*_\M$ is pointed is called a \emph{pointed module category}.

\medbreak Thus, $\C$ is group-theoretical if and only if there exist a finite group $G$ and a $3$-cocycle $\omega:G \times G \times G \to \kk^{\times}$ such that $\C$ is equivalent to the fusion category $\C(G, \omega, F, \alpha)$, where $F$ is a subgroup of $G$, and $\alpha: F \times F \to \kk^{\times}$ is a 2-cochain on $F$ satisfying $d\alpha=\omega|_{F\times F\times F}$.

\medbreak Every fusion category of FP-dimension $p^n$, for some $n \in \mathbb{N}$, is group-theoretical \cite[Corollary~9.14.16]{EGNO}; this includes the category of (co)modules over a semisimple Hopf algebra of dimension $p^n$.

 \medbreak The class of group-theoretical fusion categories is closed under categorical Morita equivalence.
In particular, it is closed under tensor products and Drinfeld centers.

\medbreak Moreover,  suppose that $\D$ is a fusion category. Then $\D$ is equivalent to a group-theoretical fusion category $\C(G, \omega, F, \alpha)$ if and only if its Drinfeld center $\mathcal Z(\D)$ is equivalent as a braided fusion category to the category of finite-dimensional representations of the {\it twisted  quantum double} $D^\omega G$ \cite{dpr}; see e.g. \cite[Theorem~1.2]{gp-ttic}.

\subsection{Fiber functors over group-theoretical fusion categories} Consider $\C=$ $\C(G, \omega, F, \alpha)$, a group-theoretical fusion category.
By \cite[Theorem~3.1]{ostrik-dd}, every indecomposable module category over $\C$ arises from a pair $(L, \beta)$, where $L$ is a subgroup of $G$ and $\beta$ is  a 2-cochain on $L$ such that $\omega\vert_{L\times L\times L} = d\beta$. The module category corresponding to the class of the pair $(L, \beta)$ is the category $\M(L, \beta) = {}_{\kk_\beta L}\C(G, \omega)_{\kk_\alpha F}$.

\medbreak Let $(L, \beta)$, $(L',\beta')$ be two such pairs. By \cite[Theorem 1.1]{pmc}, the corresponding module categories $\M(L, \beta)$ and $\M(L', \beta')$ are equivalent as $\C$-module categories if and only if
there exists an element $g \in G$ such that $L' = gLg^{-1}$ and the cohomology class of the two cocycle ${\beta'}^{-1}{\beta}^g\Omega_g$ is trivial in $H^2(L, \kk^\times)$, where
${\beta}^g$ is the 2-cochain defined by ${\beta}^g(s, t) = \beta'(gsg^{-1}, gtg^{-1})$, $s, t \in L$, and
\begin{equation*}\Omega_g(a, b) = \frac{\omega(gag^{-1}, gbg^{-1}, g) \; \omega(g, a, b)}{\omega(gag^{-1}, g, b)}, \quad a, b \in G.
\end{equation*}

\medskip

\begin{remark}\label{rmk-conjclass}
Suppose that either one of the following assumptions holds:

\smallbreak (a) The 3-cocycle $\omega$ is trivial, or

\smallbreak (b) $H^2(L, \kk^\times) = 0$.

\smallbreak

\noindent It was observed in \cite{pmc} that we obtain, as in  \cite[Theorem 3.1]{ostrik-dd}, that  $\M(L, \beta)$ and $\M(L', \beta')$ are equivalent as $\C$-module categories if and only if the pairs $(L, \beta)$ and $(L', \beta')$ are conjugated under the adjoint action of $G$.
\end{remark}

\medbreak Recall from Subsection \ref{modcat} that rank one module categories over a fusion category $\C$ correspond to fiber functors on $\C$. In the case where $\C$ is a group-theoretical fusion category, these functors were determined in \cite[Corollary 3.4]{ostrik-dd}.
 We record the corresponding classification result in the next theorem.  Recall that a 2-cocycle $\gamma$ on a finite group $S$ is \emph{non-degenerate} if the twisted group algebra $\kk_\gamma S$ is isomorphic to a matrix algebra.

\begin{theorem}{\cite[Corollary~3.4]{ostrik-dd}, \cite[Theorem 1.1]{pmc}.}\label{th-ostrik} Fiber functors on $\C(G, \omega, F, \alpha)$ correspond to pairs $(L, \beta)$, where $L$ is a subgroup of $G$ and $\beta$ is  a 2-cocycle on $L$, such that the following conditions are satisfied:
\begin{enumerate}\item[(i)] The class of $\omega\vert_{L\times L\times L}$ is trivial;
\item[(ii)] $G = LF$; and
\item[(iii)] The class of the 2-cocycle $\alpha\vert_{F\cap L}\beta^{-1}\vert_{F\cap L}$ is non-degenerate.
\end{enumerate}
Two such pairs $(L, \beta)$,  $(L', \beta')$ give rise to isomorphic fiber functors if and only if there exists an element $g \in G$ such that $L' = gLg^{-1}$ and the cohomology class of the two cocycle ${\beta'}^{-1}{\beta}^g\Omega_g$ is trivial in $H^2(L, \kk^\times)$.
\qed
\end{theorem}

\begin{remark}\label{rmk-ppp-pqq} In Sections \ref{sec-main1} and \ref{sec-main2} we shall apply Theorem \ref{th-ostrik} to determine isomorphism classes of Galois objects for families of semisimple Hopf algebras of dimensions $p^3$ and $pq^2$, respectively, where $p$ and $q$ are distinct prime numbers. As we shall see, in these contexts, either the subgroups $L$ we need to consider satisfy $H^2(L, \kk^\times) = 0$ or the 3-cocycle $\omega$ is trivial. Hence, Remark~\ref{rmk-conjclass} applies, and  the problem is reduced to determining conjugacy classes of such pairs $(L, \beta)$.
\end{remark}

%%%%%%%%%%%%%%%%%%%%%%%%%%
%%%%%%%%%%%%%%%%%%%%%%%%%%
%%%%%%%%%%%%%%%%%%%%%%%%%%
%%%%%%%%%%%%%%%%%%%%%%%%%%

\section{Abelian extensions of Hopf algebras}\label{ab-ext}
In this section we review some preliminaries on matched pairs of finite groups and Hopf algebra extensions arising from them. We refer the reader to \cite{ma-ext} and \cite{ma-newdir} for further details.

\subsection{Matched pairs of groups} \label{ss-matched}
Let $F$ and $\Gamma$ be finite groups endowed  with mutual actions by permutations
$\Gamma \overset{\vartriangleleft}\leftarrow \Gamma \times F
\overset{\vartriangleright}\to F$ such that
\begin{equation}\label{matched}
s \vartriangleright xy  = (s \vartriangleright x) ((s
\vartriangleleft x) \vartriangleright y), \quad st
\vartriangleleft x  = (s \vartriangleleft (t \vartriangleright x))
(t \vartriangleleft x), \end{equation} for all $s, t \in \Gamma$ and
$x, y \in F$. If $\lhd, \rhd$ are actions satisfying conditions \eqref{matched}, then $(F, \Gamma)$ is called a \emph{matched pair} of finite groups.

\medbreak
If $(F, \Gamma)$ is a matched pair of finite groups, then the set $F \times \Gamma$ is a group, denoted $F \bowtie \Gamma$, with multiplication defined for all $x, y \in F$, $s, t \in \Gamma$ in the form
\begin{equation*} (x, s) (y, t) = (x (s\rhd y), (s\lhd y) t).
\end{equation*}
Let us identify the subgroups $F\times 1$ and $1\times \Gamma$ of $F\bowtie \Gamma$ with  $F$ and $\Gamma$, respectively. Then the group $F \bowtie \Gamma$ admits an \emph{exact factorization} $F \bowtie \Gamma = F \Gamma$.
Conversely, every group $G$ endowed with an exact factorization into its subgroups $F' \cong F$ and $\Gamma' \cong \Gamma$ gives rise to actions by permutations  $\lhd: \Gamma \times F \to \Gamma$ and $\rhd: \Gamma \times F \to F$ making $(F, \Gamma)$ into a matched pair; these actions are determined by the relations
\begin{equation*}sx = (s \vartriangleright
x)(s \vartriangleleft x),\quad x \in F, \ \ s \in \Gamma.
\end{equation*}

\subsection{Bicrossed products  arising from matched pairs of groups} \label{bicrossed} Retain the setting of Section~\ref{ss-matched}.
 Let  $\sigma: F \times F \to (\kk^\Gamma)^\times$ and $\tau: \Gamma \times \Gamma \to (\kk^F)^\times$ be maps. Let us denote for $x, y \in F$, $s, t \in \Gamma$,
$$\sigma_s(x, y) := \sigma(x, y)(s) \quad \text{ and } \quad \tau_x(s, t) := \tau(s, t)(x).$$

Let $\kk^\Gamma \, {}^\tau\!\#_\sigma \kk F$ denote the vector space $\kk^\Gamma \otimes \kk F$ with multiplication and comultiplication defined, for all
$g,h\in \Gamma$, $x, y\in F$, by the formulas
\begin{align}\label{mult}
(e_g \# x)(e_h \# y) & =  e_g e_{h \triangleleft x^{-1}} \sigma(x,y) \# xy =~ \delta_{g \vartriangleleft x, h}\, \sigma_g(x,
y) e_g \# xy, \\
\label{delta}
\Delta(e_g \# x) &
= \sum_{st=g} \tau_x(s, t)\; e_s\# (t \vartriangleright x) \otimes e_{t}\# x.
\end{align} We shall call $\kk^\Gamma  {}^\tau\!\#_\sigma \kk F$ the \emph{bicrossed product} associated to the pair $(\sigma,  \tau)$. This bicrossed product $\kk^\Gamma  {}^\tau\!\#_\sigma \kk F$ is a Hopf algebra if and only if the pair $(\sigma, \tau)$ satisfies the following conditions, for all $x, y, z \in F$, $s,t,u \in \Gamma$
\begin{align}
\label{cociclo} \sigma_{s\lhd x}(y, z) \, \sigma_s(x, yz) &= \sigma_s(xy, z) \, \sigma_s(x, y),\\
 \label{n-sigma}\sigma_1(x,y) &= \sigma_s(x,1) = \sigma_s(1,y) = 1,\\
\label{cociclodual} \tau_{x}(st, u) \, \tau_{u \rhd x}(s, t) &= \tau_x(s, tu) \, \tau_x(t, u),\\
 \label{n-tau} \tau_1(s,t) &= \tau_x(s,1) = \tau_x(1,t) =1,
 \end{align}

 \vspace{-.3in}

 \begin{align}
\label{comp}
\sigma_{st}(x, y) \, &\tau_{xy}(s, t)  \\
&= \sigma_s(t\rhd x, (t\lhd x) \rhd y) \, \sigma_t(x, y) \,
 \tau_x(s, t) \, \tau_y(s\lhd (t\rhd x), t\lhd x). \notag
\end{align}
That is, $\kk^\Gamma  {}^\tau\!\#_\sigma \kk F$ is a Hopf algebra precisely when $\sigma, \tau$ are normalized 2-cocycles satisfying compatibility condition \eqref{comp}.

\subsection{Extensions arising from matched pairs of groups} \label{extn} Retain the setting of Section~\ref{ss-matched} and take $\kk^\Gamma
{}^{\tau}\!\#_{\sigma}\kk F$ a bicrossed product of Hopf algebras as in Section~\ref{bicrossed}.
Let $$\pi = \epsilon \otimes \id: ~\kk^\Gamma
{}^{\tau}\!\#_{\sigma}\kk F \to \kk F$$ denote the canonical projection. We
have an exact sequence of Hopf algebras
\begin{equation}\label{opext}\kk \to \kk^\Gamma \to \kk^\Gamma
{}^{\tau}\!\#_{\sigma}\kk F
\overset{\pi}\to \kk F \to \kk.\end{equation}
 Let $H$ be a Hopf algebra fitting into such an exact sequence. Then there exist mutual actions by permutations between $F$ and $\Gamma$ so we have a matched pair of groups such that $H$ is isomorphic to the bicrossed product
$\kk^\Gamma  {}^{\tau}\!\#_{\sigma}\kk F$ for appropriate compatible
actions and cocycles $\sigma$ and $\tau$. In this case, we say that $H$ is an \emph{abelian extension} associated to the matched pair $(F, \Gamma)$.

\medbreak  Let $(F, \Gamma)$ be a fixed matched pair of finite groups and let $G = F \bowtie \Gamma$.  Equivalence classes of abelian extensions associated to $(F, \Gamma)$ form an abelian group $\Opext(\kk^\Gamma, \kk F)$, whose unit element
is the class of the \emph{split} extension $\kk^\Gamma \# \kk F$.
Moreover, by a result of G. I. Kac \cite{kac}, there is an exact sequence of abelian groups
\begin{align*}
0  \to H^1(G, \kk^{\times}) &\xrightarrow{\text{Res}}   H^1(F,
\kk^{\times}) \oplus  H^1(\Gamma, \kk^{\times}) \to \Aut(\kk^\Gamma \# \kk F)
 \hspace{.1in} \to  H^2(G, \kk^{\times}) \\
 & \xrightarrow{\text{Res}}  H^2(F, \kk^{\times}) \oplus  H^2(\Gamma, \kk^{\times})
 \to \Opext(\kk^\Gamma, \kk F)  \xrightarrow{\bar\omega}  H^3(G,
\kk^{\times}) \\ & \xrightarrow{\text{Res}}  H^3(F, \kk^{\times}) \oplus  H^3(\Gamma,
\kk^{\times}) \to \dots
\end{align*}

Pertaining to the 3-cocycle $\omega \in H^3(G, \kk^{\times})$ arising from an abelian extension~\eqref{opext} in the Kac exact sequence, we have the following result.

\begin{lemma} \label{lem-trivial} Retain the notation above. Suppose that the orders of the subgroups $F$ and $\Gamma$ of $G$ are relatively prime. Then the class of $\omega$ in $H^3(G, \kk^{\times})$ is trivial.
\end{lemma}

\begin{proof} If $G$ is a finite group and  $S$ is a subgroup of $G$, then for $n \geq 1$, the corestriction (or transfer) map $\text{Cor}: H^n(S, \kk^*) \to H^n(G, \kk^*)$ has the property that the composition  $\text{Cor}\circ  \text{Res} : H^n(G, \kk^*) \to H^n(G, \kk^*)$ coincides with multiplication by the index $[G: S]$, see \cite[Chapter XII, Section 8]{CE}. By exactness of the Kac sequence, the classes of the restrictions of $\omega$ to $F$ and $\Gamma$ are both trivial. Hence  $[G~ \colon F]\omega = [G ~\colon \Gamma]\omega = 0$. Moreover, by the relative primeness condition, there exist some integers $n$ and $m$ so that
$\omega  = 1 \omega = (n[G:F] + m[G:\Gamma]) \omega = 0.$
This completes this proof.
\end{proof}

%%%%%%%%%%%%%%%%%%%%%%%%%%
%%%%%%%%%%%%%%%%%%%%%%%%%%
%%%%%%%%%%%%%%%%%%%%%%%%%%
%%%%%%%%%%%%%%%%%%%%%%%%%%

\section{Galois objects and fiber functors} \label{sec-Gal}

Here we recall how the material in the previous two sections  and the main objective of this work are connected. This connection relies on a correspondence between Galois objects of semisimple Hopf algebras described in Section~\ref{bicrossed} (or equivalently, in Section~\ref{extn}) and fiber functors of certain group-theoretical fusion categories (see Section~\ref{grpthl}).

\smallbreak Let $(F, \Gamma)$ be a matched pair of finite groups and let $\sigma: F \times F \to (\kk^\Gamma)^\times$ and \linebreak $\tau: \Gamma \times \Gamma \to (\kk^F)^\times$ be compatible cocycles.  Consider the associated bicrossed product $\kk^\Gamma  {}^{\tau}\!\#_{\sigma}\kk F$ and let $\Rep (\kk^\Gamma  {}^{\tau}\!\#_{\sigma}\kk F)$ be the fusion category  of finite-dimensional right modules over $\kk^\Gamma  {}^{\tau}\!\#_{\sigma}\kk F$.

\medbreak Then there is an equivalence of tensor categories
\begin{equation} \label{gt-equiv} \Rep (\kk^\Gamma  {}^{\tau}\!\#_{\sigma}\kk F) \cong \C(G, \omega, F, 1),
\end{equation} where $G = F \bowtie \Gamma$ and $\omega: G \times G \times G \to \kk^\times$ is a 3-cocycle representing the class of $\bar\omega(\sigma, \tau)$, where $\bar\omega: \Opext(\kk^\Gamma, \kk F) \to H^3(G, \kk^\times)$ is the map in the Kac exact sequence from \cite[Section 6.3]{schauenburg} (see also \cite[Proposition 4.3]{gp-ttic}). Explicitly, we may write
\begin{equation}\label{omega-kac}\omega (xs, yt, zu) = \sigma_s(y, t\rhd z) \tau_z(s \lhd y, t),\quad x, y, z \in F, \ \ s, t, u \in \Gamma.
\end{equation}

\smallbreak As a consequence of this fact we obtain the following parameterization of right Galois objects for an abelian extension.

\begin{proposition} \label{bijcorr} Let $H$ be a semisimple Hopf algebra fitting into an exact sequence $\kk \to \kk^\Gamma  \to H \to \kk F \to \kk$. Let also $(\sigma, \tau) \in \Opext (\kk^\Gamma, \kk F)$ such that $H$ is isomorphic to the bicrossed product $\kk^\Gamma {}^\tau\!\#_\sigma \kk F$ and $\omega = \bar \omega(\sigma, \tau)$ be the 3-cocycle given by \eqref{omega-kac}. Then there is a bijective correspondence between
\begin{enumerate}
\item[(i)] Isomorphism classes of right Galois objects of $H^*$.
\item[(ii)] Isomorphism classes of fiber functors on the category $\C(F \bowtie \Gamma, \omega, F, 1)$.
\end{enumerate}
\end{proposition}

\begin{proof}
By a result of Ulbrich \cite{ulbrich}, isomorphism classes of right Galois objects of $(\kk^\Gamma  {}^{\tau}\!\#_{\sigma}\kk F)^*$ are in bijective correspondence with isomorphism classes of fiber functors on the category $(\kk^\Gamma{}^{\tau}\!\#_{\sigma}\kk F)^*\comod \cong \Rep (\kk^\Gamma{}^{\tau}\!\#_{\sigma}\kk F)$. The tensor equivalence~\eqref{gt-equiv} now implies the result.
\end{proof}

%%%%%%%%%%%%%%%%%%%%%%%%%%
%%%%%%%%%%%%%%%%%%%%%%%%%%
%%%%%%%%%%%%%%%%%%%%%%%%%%
%%%%%%%%%%%%%%%%%%%%%%%%%%

\section{Semisimple Hopf algebras of dimension $p^3$}\label{s-ppp}

In this section we study the first class of noncommutative, noncocommutative, semisimple Hopf algebras that are of interest to this work.

\medbreak Let $p$ be an odd prime number. Let also $$F = \langle a, b: \, a^p = b^p = 1, ab = ba\rangle \cong \Z_p \times \Z_p \quad \text{and} \quad\Gamma = \langle x: \, x^p = 1 \rangle \cong \Z_p,$$

 Consider the action by Hopf algebra automorphisms $\rightharpoonup: \kk\Gamma \otimes \kk^F \to \kk^F$ defined in the form
\begin{equation*} (x \rightharpoonup f)(a^ib^j) = f(a^{i+j}b^j), \quad 0\leq i, j \leq p-1.
\end{equation*}

Let $\zeta \in \kk$ be a $p$th root of $1$ and let $g \in  G(\kk^F)$ be a $\Gamma$-invariant group-like element.
Masuoka constructed in \cite{masuoka-pp} a self-dual semisimple Hopf algebra $A_{\zeta, g}$ of dimension $p^3$ that fits into an abelian extension $\kk \to \kk^F \to A_{\zeta, g} \to \kk\Gamma \to \kk$.

\begin{definition}[$A_{\zeta, g}$] \label{def-A} The Hopf algebra $A_{\zeta, g}$ is defined as the $\kk^F$-ring generated by an element $\bar x$ with relations
\begin{equation}\label{rels-a} \bar x^p = g, \quad \bar x f = (x \rightharpoonup f) \bar x,
\end{equation}
for all $f \in \kk^F$.
The coalgebra structure of $A_{\zeta, g}$ is determined by the requirements that $\kk^F$ is a subcoalgebra and
\begin{equation}\label{com-a}\Delta(\bar x) = \sum_{i, j, r, \ell} \zeta^{jr} \, e_{i,j} \bar x \otimes e_{r,\ell} \bar x, \quad \epsilon(\bar x) = 1,
\end{equation} where, for all $0 \leq i, j \leq p-1$, $e_{i,j} \in \kk^F$ is defined as $e_{i,j} (a^kb^\ell) = \delta_{i, k}\delta_{j, \ell}$.
\end{definition}

Moreover, we have the following classification result:

\begin{theorem}\label{ppp}\cite[Theorem 3.1]{masuoka-pp} Let $H$ be a noncommutative, noncocommutative, semisimple Hopf algebra of dimension $p^3$. Then $H$ is isomorphic to precisely one of the Hopf algebras
$$A_{\zeta, 1}, ~~~ A_{\zeta^t, 1}, ~~~ A_{\zeta, g}, ~~~\dots, ~~~A_{\zeta^{p-1}, g},$$
where $\zeta \in \kk$ is a fixed primitive $p$th root of unity, $1 \neq g \in G(\kk^F)$ is a $\Gamma$-invariant group-like element of $\kk^F$, and $t \in \F_p$ is a fixed quadratic nonresidue.  \qed \end{theorem}

\begin{lemma}\label{a-cop} Let $A_{\zeta, g}$ be the Hopf algebra in Definition \ref{def-A}. Then there are isomorphisms of Hopf algebras $A_{\zeta, g}^{cop} \cong A_{\zeta, g}^{op} \cong A_{\zeta^{-1}, g}$.
\end{lemma}

\begin{proof} We only need to prove the second isomorphism. The Hopf algebras $A_{1, g}$ are cocommutative, so the claim is evident if $\zeta = 1$.
	
\medbreak Suppose that $\zeta \neq 1$. By \eqref{rels-a}, the following relations hold in $A_{\zeta, g}$, for all $0 \leq i, j \leq p-1$:
$$\bar x e_{i,j} = e_{i-j, j} \, \bar x.$$
Hence, $\bar x \, ._{op} \, e_{i,-j} = e_{i-j,-j} \, ._{op} \, \bar x$ in $A_{\zeta, g}^{op}$, for all $0 \leq i, j \leq p-1$.

From the definition of $A_{\zeta, g}$, we find that there exists a unique algebra map $\gamma: A_{\zeta^{-1}\!, g} \to A_{\zeta, g}^{op}$, such that
$$\gamma(\bar x) = \bar x, \qquad \gamma(e_{i, j}) = e_{i, -j}, \quad  0\leq i, j \leq p-1.$$
Moreover, $\gamma\vert_{\kk^F}$ is a Hopf algebra isomorphism. In addition, \eqref{com-a} implies that  $\Delta(\gamma(\bar x)) = (\gamma \otimes \gamma) \Delta(\bar x)$. Thus $\gamma$ is a bialgebra isomorphism, and hence a Hopf algebra isomorphism.	
\end{proof}

Consider the matched pair $(F, \Gamma)$, where the action $\lhd: \Gamma \times F \to \Gamma$ is trivial and the action  $\rhd: \Gamma \times F \to F$ is the action by group automorphisms determined by
\begin{equation}\label{action-ppp} x \rhd a = a, \quad x \rhd b = ab.
\end{equation}

The associated group $G = F \bowtie \Gamma$ coincides with the semidirect product $F \rtimes \Gamma$. Note that there is an isomorphism of groups $G \cong \UT(3, p)$ of upper triangular unipotent $3\times 3$ matrices with entries in the field $\F_p$ with $p$ elements.

\medbreak Let $\zeta, \lambda \in \kk$ be $p$th roots of unity. Let $\sigma: F \times F \to (\kk^\Gamma)^\times$ and $\tau: \Gamma \times \Gamma \to (\kk^F)^\times$ be the maps defined by
\begin{equation}\label{sigma-tau}  \sigma_{x^n}(a^ib^j, a^{i'}b^{j'}) = \zeta^{-nji' - \binom{n}{2}jj'},  \quad
\tau_{a^ib^j}(x^n, x^m) = \lambda^{j [\frac{n+m}{p}]},
\end{equation}
for all $0 \leq n, m, i, j, i', j' \leq p-1$, where $\binom{n}{2}$ denotes the quotient $\frac{n(n-1)}{2}$, and $0 \leq [\frac{N}{p}] \leq p-1$ denotes the largest integer less than $N/p$ (that is, the Gauss symbol).

\smallbreak  Observe that, for all $0 \leq n, m, i, j \leq p-1$, we have
\begin{equation}\label{alt-tau} \tau_{a^ib^j}(x^n, x^m) = \begin{cases}1, \quad \text{\, if } n+m < p,\\
\lambda^j, \quad \text{if } n+m \geq p. \end{cases}
\end{equation}

\begin{proposition}[$H_{\zeta, \lambda}$]\label{iso} Retain the notation above. Let $\zeta, \lambda$ be $p$th roots of $1$ and let  $\sigma$ and $\tau$ be given by \eqref{sigma-tau}.  Then the bicrossed product $$H_{\zeta, \lambda} := \kk^\Gamma {}^\tau\!\#_{\sigma}\kk F$$ is a Hopf algebra.
Further, there is an isomorphism of Hopf algebras $H_{\zeta, \lambda}^* \cong A_{\zeta, g}$, where $g \in G(\kk^F)$ is given by $g(a^i b^j) = \lambda^j$, for all $0 \leq i, j \leq p-1$.
\end{proposition}

\begin{proof} It is straightforward to check that $\sigma$ and $\tau$ satisfy conditions \eqref{cociclo}--\eqref{comp} in Section~\ref{bicrossed} with respect to the matched pair $(F, \Gamma)$. Therefore $H = \kk^\Gamma {}^\tau\!\#_{\sigma}\kk F$ is a Hopf algebra.

\medbreak Consider the matched pair $(\Gamma, F)$ where the action $\rdual: F \times \Gamma \to \Gamma$ is trivial and the action  $\ldual: F \times \Gamma \to F$ is given by
\begin{equation}\label{dualaction-ppp} a \ldual x = a, \quad b \ldual x = ab.
\end{equation}
A direct computation shows that $\tau$ and $\sigma$ satisfy conditions \eqref{cociclo}--\eqref{comp} with respect to this matched pair (note that the roles of $\sigma$ and $\tau$ are reversed in this context). In this way we get another Hopf algebra $\kk^F{}^\sigma\!\#_{\tau}\kk \Gamma$.

\medbreak It follows from Formulas \eqref{mult} and \eqref{delta} that the map $$\langle \; , \; \rangle : \left(\kk^F{}^\sigma\!\#_{\tau}\kk \Gamma \right) \otimes \left(\kk^\Gamma {}^\tau\!\#_{\sigma}\kk F\right)^{cop} \to \kk$$ given by
$$\langle e_{i,j} \# x^n, e_{x^m} \# a^{i'}b^{j'} \rangle = \delta_{n, m} \; \delta_{i, i'}\; \delta_{j, j'}, \quad 0 \leq i, i', j, j', n, m \leq p-1,$$
defines a non-degenerate bialgebra pairing. Thus we obtain  an isomorphism of Hopf algebras
\begin{equation}\label{dual-cop}\left(\kk^\Gamma {}^\tau\!\#_{\sigma}\kk F\right)^* \cong \left(\kk^F{}^\sigma\!\#_{\tau}\kk \Gamma \right)^{cop}. \end{equation}

\medbreak Suppose that  $g \in G(\kk^F)$ is given by $g(a^i b^j) = \lambda^j$, for all $0 \leq i, j \leq p-1$. Then $g$ is a  $\Gamma$-invariant group-like element of $\kk^F$ with respect to the action induced by~\eqref{dualaction-ppp}.

\medbreak Let $\mathfrak{X} := 1\# x = \sum_{i,j} e_{i,j} \# x \in \kk^F{}^\sigma\!\#_{\tau}\kk \Gamma$. The expression \eqref{alt-tau} for the cocycle $\tau$ implies that
$$\tau(x, x^\ell) = \sum_{i,j} \tau_{a^ib^j}(x,x^\ell)e_{i,j} = \sum_{i,j} e_{i,j} =  1_{\kk^F},$$ for all $0 \leq \ell < p-1$.
From Formula \eqref{mult} for the multiplication in $\kk^F{}^\sigma\!\#_{\tau}\kk \Gamma$, we obtain
$$\mathfrak{X}^p = \tau(x, x)\, \tau(x, x^2)\, \dots \tau(x, x^{p-1})  \# x^{p} = \tau(x, x^{p-1}) \# 1 = \tau(x, x^{p-1}) = g.$$
It follows from Definition \ref{def-A} that there is a unique homomorphism of algebras $\varphi: A_{\zeta^{-1}, g} \to \kk^F{}^\sigma\!\#_{\tau}\kk \Gamma$ such that $\varphi(f) = f$, for all $f \in \kk^F$, and $\varphi(\bar x) = \mathfrak{X}$.

\medbreak From Formula \eqref{delta} for the comultiplication in $\kk^F{}^\sigma\!\#_{\tau}\kk \Gamma$ we find that
$$\Delta(\mathfrak{X}) = \sum_{i, j, r, \ell}  \sigma_{x}(a^ib^j, a^r b^\ell) \; \left(e_{i, j}\# x\right) \otimes \left(e_{r, \ell} \# x\right) = \sum_{i, j, r, \ell} \zeta^{-jr} \; \left(e_{i, j}\# x\right) \otimes \left(e_{r, \ell} \# x\right).$$ Hence $\Delta(\varphi(\bar x)) = (\varphi \otimes \varphi)\Delta(\bar x)$. Since also
$\Delta(\varphi(f)) = (\varphi \otimes \varphi)\Delta(f)$, for all $f \in \kk^F$, and $\kk^F$ and $\bar x$ generate $A_{\zeta^{-1}, g}$ as an algebra,  we get that $\varphi$ is a bialgebra map and therefore it is a Hopf algebra map. Moreover, $\varphi$ is surjective because $\kk^F$ and $\mathfrak{X}$ generate $\kk^F{}^\sigma\!\#_{\tau}\kk \Gamma$. Therefore $\varphi$ is an isomorphism, since $\dim A_{\zeta^{-1}, g} = p^3 = \dim \kk^F{}^\sigma\!\#_{\tau}\kk \Gamma$.

\medbreak Thus $\kk^F{}^\sigma\!\#_{\tau}\kk \Gamma \cong A_{\zeta^{-1}, g}$ as Hopf algebras. Combining this isomorphism with the one in \eqref{dual-cop} and Lemma \ref{a-cop}, we find that $H_{\zeta, \lambda}^* = \left(\kk^\Gamma {}^\tau\!\#_{\sigma}\kk F\right)^* \cong A_{\zeta, g}$, as claimed.
\end{proof}

Now combining Theorem \ref{ppp} and Proposition \ref{iso} we obtain:

\begin{corollary} Every noncommutative, noncocommutative, semisimple Hopf algebra of dimension $p^3$ is isomorphic to precisely one of the bicrossed products $$H_{\zeta, 1},~~~ H_{\zeta^t, 1}, ~~~H_{\zeta, \lambda},~~~ \dots, ~~~H_{\zeta^{p-1}, \lambda},$$
where $\zeta, \lambda \in \kk$ are primitive $p$th roots of unity   and $t \in \F_p$ is  a quadratic nonresidue.   \qed
\end{corollary}

Next, we consider the representation categories of the Hopf algebras above.

\begin{proposition}\label{cat-eq} Let $G = F \rtimes \Gamma \cong \UT(3, p)$ and let $\zeta, \lambda \in \kk$ be $p$th roots of unity. Then there is an equivalence of tensor categories
\begin{equation*}\Rep H_{\zeta,\lambda} ~~\cong ~~ \C(G, \omega_{\zeta, \lambda}, F, 1), \end{equation*}
where $\omega_{\zeta, \lambda}: G \times G \times G \to \kk^\times$ is the 3-cocycle given by the formula
\begin{equation}\label{omega}\omega_{\zeta, \lambda} (a^ib^jx^n, a^{i'}b^{j'}x^{n'}, a^{i''}b^{j''}x^{n''}) =  \zeta^{-nj'(i''+n' j'') - \binom{n}{2} j' j''}  \; \lambda^{j''[\frac{n+n'}{p}]},
\end{equation}
for all $0 \leq i, i', i'', j, j', j'', n, n', n'' \leq p-1$.
\end{proposition}

\begin{proof}
By \cite[Proposition 4.3 and Formula (3.12)]{gp-ttic} there is an equivalence of tensor categories
$\Rep H_{\zeta,\lambda} ~~\cong ~~ \C(G, \omega_{\zeta, \lambda}, F, 1)$, where $G = F \rtimes \Gamma$ and $\omega_{\zeta, \lambda}$   is the 3-cocycle on $G$  given by
\begin{equation*}\omega_{\zeta, \lambda} (a^ib^jx^n, a^{i'}b^{j'}x^{n'}, a^{i''}b^{j''}x^{n''}) = \sigma_{x^n}(a^{i'}b^{j'}, x^{n'}\rhd a^{i''}b^{j''}) \tau_{a^{i''}b^{j''}}(x^n, x^{n'}),\end{equation*}
for all $0 \leq i, i', i'', j, j', j'', n, n', n'' \leq p-1$.
It follows from \eqref{action-ppp} that for all $0 \leq n', i'', j'' \leq p-1$, we have
$x^{n'}\rhd a^{i''}b^{j''} = a^{i''+n'j''}b^{j''}$. Therefore, from \eqref{sigma-tau}, we obtain the expression \eqref{omega} for $\omega_{\zeta, \lambda}$.
This proves the proposition.
\end{proof}

We also obtain the following consequence of  Propositions~\ref{iso} and~\ref{cat-eq}.

\begin{corollary}\label{comod-a} Let $\zeta \in \kk$ be a primitive $p$th root of unity and let  $g$  be a $\Gamma$-invariant group-like element  of $\kk^F$. Let also  $A_{\zeta, g}$ be one of the noncommutative, noncocommutative, semisimple Hopf algebras of dimension $p^3$ and take $\lambda = g(b)$. Then there is an equivalence of tensor categories
$$A_{\zeta, g}\comod ~~\cong ~~ \C(G, \omega_{\zeta, \lambda}, F, 1).$$
\end{corollary}

\begin{proof} By \cite[Theorem~2.18]{masuoka-pp} and Proposition~\ref{iso}, we have isomorphisms of Hopf algebras
\begin{equation*} \label{eq:isoms}
A_{\zeta, g}~~~ \cong ~~~A_{\zeta, g}^* ~~~\cong ~~~H_{\zeta, \lambda} ~~~=~~~ \kk^F {}^\tau\!\#_\sigma \kk\Gamma.
\end{equation*}
Hence, $A_{\zeta, g}\comod \cong \Rep A_{\zeta, g}^* \cong \Rep H_{\zeta, \lambda} \cong \C(G, \omega_{\zeta, \lambda}, F, 1)$ as tensor categories, by Proposition~\ref{cat-eq}.
\end{proof}

%%%%%%%%%%%%%%%%%%%%%%%%%%
%%%%%%%%%%%%%%%%%%%%%%%%%%
%%%%%%%%%%%%%%%%%%%%%%%%%%
%%%%%%%%%%%%%%%%%%%%%%%%%%

\section{Galois objects for semisimple Hopf algebras of dimension $p^3$} \label{sec-main1}

In this section, we compute the number of right Galois objects of the bicrossed products $H_{\zeta, \lambda}$ of Proposition~\ref{iso}, and thus of Masuoka's Hopf algebras $A_{\zeta, g}$ of Definition~\ref{def-A}.

\smallbreak  To begin we recall facts about 3-cocycles on cyclic groups.
Let $N \geq 1$ be an integer and let $L = \langle c\rangle$ be a cyclic group of order $N$. Then, for every $N$th root of unity $\theta$, the expression
\begin{equation}\label{theta}
\omega_\theta(c^n, c^m, c^\ell) = \theta^{\ell[\frac{n+m}{N}]}, \quad 0\leq n, m, \ell \leq N-1,
\end{equation} defines a 3-cocycle on $L$. Moreover the map
\begin{equation} \label{H3-iso}
\mathbb G_N \to H^3(L, \kk^\times), \quad \theta \mapsto [\omega_\theta],
\end{equation}
 is a group isomorphism, where $\G_N$ denotes the multiplicative group of $N$th roots of unity in $\kk$ and $[\omega_\theta]$ denotes the class of $\omega_\theta$ in $H^3(L, \kk^\times)$.

\begin{lemma}\label{clase} Let $L = \langle c\rangle$ be a cyclic group of order $N$. Then, for every $N$th root of unity $\xi$, the expression
\begin{equation}\label{tom}\tilde \omega_\xi(c^n, c^m, c^\ell) = \xi^{m\left(n\binom{\ell}{2} + \ell\binom{n}{2} + nm\ell\right)}, \quad 0\leq n, m, \ell \leq N-1,
\end{equation} defines a 3-cocycle on $L$. Further, the 3-cocycle $\tilde \omega_\xi$ is cohomologous to $\omega_{\xi^d}$ defined in \eqref{theta}, where $d = \frac{1}{6}(N-1)N(2N-1)$.
In particular, if $N$ is odd and not divisible by~$3$, then the class of $\tilde \omega_\xi$ is trivial.
\end{lemma}

\begin{proof} It is straightforward to verify that $\tilde \omega_\xi$ in \eqref{tom} is a 3-cocycle. Hence there must exist an $N$th root of unity $\theta$ such that $\tilde \omega_\xi$ is cohomologous to $\omega_\theta$ in~\eqref{theta} by~\eqref{H3-iso}. So, there exists a 2-cochain $t: L \times L \to \kk^\times$ such that $\tilde \omega_\xi \,  dt  = \omega_\theta$, where
$$d t(c^n, c^m, c^\ell) = \frac{t(c^n, c^{m+\ell}) \, t(c^m, c^\ell)}{t(c^{n+m}, c^\ell) \, t(c^n, c^m)}, \quad 0\leq n, m, \ell \leq N-1.$$ We may assume that $t$ is a normalized 2-cochain, that is, $t(1, c^n) = 1 = t(c^n, 1)$, for all $0 \leq n \leq N-1$. Observe that, for all $0\leq n, m, \ell \leq N-1$,
\begin{equation*}\omega_\theta(c^n, c^m, c^\ell) = \begin{cases}1, \quad \text{ if } n+m < N, \\
\theta^\ell, \quad \text{if } n+m \geq N.\end{cases}
\end{equation*}
Therefore we obtain the following relation, for all $0\leq n, m, \ell\leq N-1$:
\begin{equation*} \tilde \omega_\xi(c^n, c^m, c^\ell) \;  dt(c^n, c^m, c^\ell) = \begin{cases}1, \qquad \text{if } n+m < N, \\
\theta^\ell, \quad \text{\, if } n+m \geq N.\end{cases}
\end{equation*}

Combined with \eqref{tom}, the last relation implies that, for all $0 \leq m < N-1$,
\begin{equation}\label{nl1}\xi^{m^2} = \frac{t(c^{m+1}, c) \, t(c, c^m)}{t(c, c^{m+1}) \, t(c^m, c)},
\end{equation}
and also that
\begin{equation}\label{N-1}\xi = \xi^{(N-1)^2} =  \theta \; \frac{t(c^{N}, c) \, t(c, c^{N-1})}{t(c, c^{N}) \, t(c^{N-1}, c)} = \theta \, \frac{t(c, c^{N-1})}{t(c^{N-1}, c)}.
\end{equation}

An inductive argument, using Formula \eqref{nl1}, shows that
\begin{equation*}\frac{t(c^{m+1}, c)}{t(c, c^{m+1})} = \xi^{\sum_{j = 1}^m j^2}, \qquad 0 < m < N-1.
\end{equation*}
Hence, from \eqref{N-1},
\begin{equation*}\xi = \xi^{(N-1)^2} = \theta \, \frac{t(c, c^{N-1})}{t(c^{N-1}, c)} = \theta \,  \xi^{-\sum_{j = 1}^{N-2} j^2}.
\end{equation*}
Therefore $\theta = \xi^{\sum_{j = 1}^{N-1} j^2} = \xi^{\frac{(N-1)N(2N-1)}{6}}$, as claimed.
Finally, note that if $N$ is odd and not divisible by $3$, then $\theta = (\xi^N)^{\frac{(N-1)(2N-1)}{6}} = 1$. This finishes the proof of the lemma.
\end{proof}

This brings us to the main result of this section.

\begin{theorem} \label{mainGalois} Let $p$ be an odd prime number. Let $\zeta \in \kk$ be a primitive $p$th root of unity,   $g \in \kk^F$ a $\Gamma$-invariant group-like element and $t \in \mathbb F_p$ a quadratic nonresidue. Fix  one of the noncommutative, noncocommutative, semisimple Hopf algebras $A_{\zeta, g}$ of dimension $p^3$. Then the following statements hold.
\begin{enumerate}
\item Suppose that $p=3$. Then any right Galois object of the Hopf
algebras $A_{\zeta,1}$, $A_{\zeta^2,1}$ must be trivial. Also, the
Hopf algebras $A_{\zeta,g}$, $A_{\zeta^2,g}$ each have exactly two
right Galois objects up to isomorphism.

\smallskip

\item Suppose that $p>3$. Then any right Galois object of the Hopf
algebras $A_{\zeta,g}, \dots, A_{\zeta^{p-1},g}$ must be trivial.
Also, the Hopf algebras $A_{\zeta,1}$, $A_{\zeta^t,1}$ each have
exactly $p$ right Galois objects up to isomorphism.

\end{enumerate}
In addition, if $R$ is a right Galois object for the Hopf algebra $A_{\zeta, g}$, then the left Galois Hopf algebra $L(R, A_{\zeta, g})$ is isomorphic to $A_{\zeta, g}$.
\end{theorem}

\begin{proof} Let $\lambda = g(b)$. By Corollary \ref{comod-a}, there is an equivalence of tensor categories
$$A_{\zeta, g}\comod ~~\cong ~~ \C(G, \omega, F, 1),$$ and
the 3-cocycle $\omega = \omega_{\zeta, \lambda}$ is given, for all $0 \leq i, i', i'', j, j', j'', n, n', n'' \leq p-1$, by
\begin{equation*}\omega (a^ib^jx^n, a^{i'}b^{j'}x^{n'}, a^{i''}b^{j''}x^{n''}) =  \zeta^{-nj'(i''+n' j'') - \binom{n}{2} j' j''}\; \lambda^{j''[\frac{n+n'}{p}]}.
\end{equation*}

\medbreak
Now by Proposition~\ref{bijcorr}, right Galois objects of $A_{\zeta, g}$ correspond to fiber functors of the category
$A_{\zeta, g}\comod$. In view of the previous equivalence and  Theorem~\ref{th-ostrik}, we need to determine the pairs $(L, \beta)$ where $L$ is a subgroup of $G$ and $\beta$ is  a 2-cocycle on $L$, such that the following conditions are satisfied:
\smallskip

\begin{enumerate}
\item[(i)] The class of $\omega\vert_{L\times L\times L}$ is trivial.
\item[(ii)] $G = LF$.
\item[(iii)] The class of the 2-cocycle $\beta\vert_{F\cap L}$ is non-degenerate.
\end{enumerate}
\smallskip

Suppose $(L, \beta)$ is such a pair.
Conditions (i)-(iii) imply that $L$ is a subgroup of order $p$ of $G = F \rtimes \Gamma$, which is not contained in $F$. Indeed, condition (iii) implies that $|F \cap L| =1$ or $p^2$. If $|F \cap L|=p^2$, then either $L =F$ or $L=G$; the first case contradicts (ii) and the second case contradicts (i) (see Corollary~\ref{class-omega}). So $|F \cap L|=1$ and $G = LF$ is an exact factorization.
In particular, the subgroup $L$ is cyclic.  In view of Remark \ref{rmk-conjclass}, isomorphism classes of fiber functors in $A_{\zeta, g}\comod$ are classified by conjugacy classes of such pairs $(L, \beta)$.

\medbreak Recall that $G \cong \text{UT}(3,p)$; it is well-known that $\text{UT}(3,p)$ has $p$ conjugacy classes of subgroups of order $p$ not contained in $F$. A straight-forward computation then shows that conjugacy classes of such subgroups $L$ of $G$ are represented by the subgroups $$L_j = \langle b^jx\rangle, \quad 0 \leq j \leq p-1.$$
One could check this by using, for all $0 \leq n \leq p-1$, the formula
\begin{equation} \label{powers}
(b^jx)^n = a^{\binom{n}{2}j}b^{jn}x^n.
\end{equation}

\medbreak Moreover, since $|L_j|=p$, we have that $\beta \in H^2(L_j, \kk^\times)$ is trivial. So we have accounted for conditions (ii) and (iii) and now it suffices to verify condition~(i) for the pairs $(L_j,1)$ to parameterize right Galois objects of $A_{\zeta,g}$.

\medbreak
Observe that, for $j = 0$, $L_0 = \Gamma$ and $\omega\vert_{\Gamma \times \Gamma \times \Gamma} = 1$. The pair $(\Gamma, 1)$ corresponds to the trivial right Galois object.

\medbreak  Now by \eqref{omega} and \eqref{powers}, the restriction of the 3-cocycle $\omega$ to the subgroup $L_j$ is determined by the formula, for all $0\leq n, m, \ell \leq p-1$
\begin{equation*}\omega((b^jx)^n, (b^jx)^m, (b^jx)^\ell) = \zeta^{- j^2m\left(n\binom{\ell}{2} + \ell\binom{n}{2} + nm\ell\right)} \, \lambda^{j \ell[\frac{n+m}{p}]}.
\end{equation*}
It follows from Lemma \ref{clase} that $\omega\vert_{L_j}$ is cohomologous to the 3-cocycle $$\omega_{\zeta^{ - j^2\frac{(p-1)p(2p-1)}{6}}\, \lambda^j} = \begin{cases}\omega_{\zeta^{ j^2 }\, \lambda^j}, \quad \text{\, if } p = 3, \\
\omega_{\lambda^j}, \qquad \text{\quad if } p > 3. \end{cases}$$

\medbreak  Suppose that $j \neq 0$ and $p = 3$. Then the class of $\omega\vert_{L_j \times L_j \times L_j}$ is not trivial if $\lambda = 1$. That is, condition~(i) is violated for pairs $(L_j, 1)$ for $A_{\zeta, 1}$ (and for $A_{\zeta^2, 1}$) and $(\Gamma, 1)$ represents the unique conjugacy class of pairs $(L, \beta)$ giving rise to a fiber functor in this case. On the other hand, for any $\lambda \neq 1$, there exists a unique $j \neq 0$, such that the class of $\omega\vert_{L_j \times L_j \times L_j}$ is trivial. This implies that there are two conjugacy classes of pairs $(L, \beta)$ corresponding to fiber functors of $A_{\zeta, g}\comod$ when $g \neq 1$. Thus part~(1) holds.

\medbreak In the remaining case when $j \neq 0$ and $p > 3$, the class of the cocycle $\omega\vert_{L_j}$ is trivial if and only if $\lambda = 1$. This implies part~(2).

\medbreak The last statement of the theorem follows from Theorem \ref{cls-morita}  (see Proposition~\ref{prop-connect}  (1), (2)  and Proposition ~\ref{cocycle-ppp} later in the paper).
\end{proof}

\begin{remark} \label{referee} 
Suppose that $p > 3$ and let $H$ be one of the Hopf algebras $A_{\zeta, 1}$, $A_{\zeta^t, 1}$. Thus $H$ fits into a central extension $\kk \to \kk^\Gamma \to H \to \kk F \to \kk$, where $\Gamma = \mathbb Z_p$ and $F = \mathbb Z_p \times \mathbb Z_p$.  It follows from Theorem 6.7(2) that the number of right Galois objects of $H$ coincides with the order of the second cohomology $H^2(F,\kk^\times)$.

\medbreak
This poses the question of deciding if every Hopf 2-cocycle for $H$ arises from some group 2-cocycle for $F$ via the map $H \to \kk F$ or, equivalently, if for every $H$-Galois object $R$ there exists some  $\kk F$-Galois object $Z$ such that $R = Z \Box_{\kk F}H$, which is referred to as the {\it inflation} of the Galois object $Z$ in \cite{ma-contemp}.

\medbreak
The following is a possible approach to this question that was kindly pointed out to us by the referee. Let $\phi : H \to R$ be a unit-preserving $H$-comodule map which is convolution invertible,
or equivalently, bijective. By \cite[Theorem 4]{Gunther}, the answer is positive for a given $H$-Galois object $R$ if and only if $\phi(\kk^\Gamma)$ is central in $R$.

\medbreak
The developments required to conduct this approach exceed however the scope of the present paper, and we postpone them for a future work.
\end{remark}

%%%%%%%%%%%%%%%%%%%%%%%%%%
%%%%%%%%%%%%%%%%%%%%%%%%%%
%%%%%%%%%%%%%%%%%%%%%%%%%%
%%%%%%%%%%%%%%%%%%%%%%%%%%

\section{Semisimple Hopf algebras of dimension $pq^2$}\label{s-pqq}

In this section we introduce and describe the second class of noncommutative, noncocommutative, semisimple Hopf algebras that are of interest to this work.  Let $p$ and $q$ be distinct  prime  numbers.

\medbreak

Suppose that $q\equiv1$ mod $p$ and let  $m$ a primitive $p$th root of $1$ modulo $q$.  Consider the matched pair $(F, \Gamma)$, where
$$\Gamma = \langle a, b: \, a^q = b^q = 1, ~ab = ba\rangle \cong \Z_q \times \Z_q, \ \ F = \langle g: \, g^p = 1 \rangle \cong \Z_p,$$
such that the action $\rhd: \Gamma \times F \to F$ is trivial and $\lhd: \Gamma \times F \to \Gamma$  is the action by group automorphisms determined by
\begin{equation}\label{action-pqq} a \lhd g^{-1} = a^m, \quad b \lhd g^{-1} = b^{m^\lambda},
\end{equation}
for $0\leq \lambda \leq p-1$ an integer such that $\lambda\not \equiv -1~\text{mod } p$.
The associated group $G = F \bowtie \Gamma$ coincides with the semidirect product $F \ltimes \Gamma$ with respect to the action~\eqref{action-pqq}.

\medbreak Let $\zeta \in \kk$ be a primitive $q$th root of $1$. The third author described in \cite[Section~1.4]{natale-pqq} a semisimple Hopf algebra $B_\lambda(m,\zeta)$ of dimension $pq^2$ that fits into an abelian exact sequence $\kk \to \kk^\Gamma \to B_\lambda(m,\zeta) \to \kk F \to \kk$.

\begin{definition}[$B_{\lambda}$] \label{B-lambda}The Hopf algebra $B_\lambda(m,\zeta)$ is defined as $\kk^\Gamma{}^{\tau}\!\#_{\rightharpoonup}\kk F$ where $\rightharpoonup$ is induced by $\lhd$ and the cohomology class of the dual cocycle can be represented by a factor set $\tau:\Gamma\times\Gamma\to (\kk^F)^{\times}$ of the form $$\tau(a^ib^j,a^kb^l)=u^{jk}$$ for all $0\leq i,j,k,l \leq q-1$ and $u=\sum_{t=0}^{p-1}\zeta_t e_{g^t}\in (\kk^F)^\times$. Here, $e_{g^t}$ is the dual basis element of $g^t \in \kk F$.  Further, $u^q =1$,  $\zeta_1=\zeta$ and $\zeta_t=\zeta^{c_t(m^{\lambda+1})}$ for all $0\leq t\leq p-1$ where $c_t(n):=1+n+\cdots  +n^{t-1}$, for $n \in \mathbb{Z}$. For $m,\zeta$ fixed, denote $$B_\lambda:=B_\lambda(m,\zeta).$$
\end{definition}

The following result will be needed in the proof of Theorem \ref{main-teo}.

\begin{lemma}\label{h2-g} With the notation above, we have $H^2(F \ltimes \Gamma, \kk^\times) = 0$.
\end{lemma}

\begin{proof} Since the orders of $F$ and $\Gamma$ are relatively prime, $H^1(F, H^1(\Gamma, \kk^\times)) = H^2(F, H^1(\Gamma, \kk^\times)) = 0$. Also, $H^2(F, \kk^\times) =0$, so the kernel of the restriction map from $H^2(F \ltimes \Gamma, \kk^\times)$ to $H^2(\Gamma, \kk^\times)$ is $H^2(F \ltimes \Gamma, \kk^\times)$. By \cite[Theorem 2]{tahara}, this restriction map induces a group isomorphism
$H^2(F \ltimes \Gamma, \kk^\times) \cong H^2(\Gamma, \kk^\times)^F$, where the action by group automorphisms of $F$ on $H^2(\Gamma, \kk^\times)$ is the one arising from the adjoint action of $F$ on $\Gamma$, or in other words, from the action \eqref{action-pqq}.

\medbreak For each $q$th root of unity $\xi \in \kk$, let $\alpha_\xi: \Gamma \times \Gamma \to \kk^\times$ be the 2-cocycle by $$\alpha_\xi(a^ib^j, a^{i'}b^{j'}) = \xi^{ji'}, \quad 0 \leq i, i', j, j' \leq p-1.$$
Then the map $\xi \mapsto \alpha_\xi$ induces a group isomorphism $\mathbb G_q \cong H^2(\Gamma, \kk^\times)$, where $\mathbb G_q$ is the group of $q$th roots of unity in $\kk$.
A direct computation shows that the action of $F$ on $H^2(\Gamma, \kk^\times)$ induces the action  $\rightharpoonup: F\times \mathbb G_q \to \mathbb G_q$, determined by $g \rightharpoonup \xi = \xi^{m^{\lambda+1}}$, for all $\xi \in \mathbb G_q$.

\medbreak Since, by assumption, $m$ is of order $p$ in $\mathbb F_q^\times$, and $\lambda \not \equiv -1~\text{mod } p$, then the action of $F$ on $H^2(\Gamma, \kk^\times)$ has no non-trivial fixed points. Therefore $H^2(F \ltimes \Gamma, \kk^\times) \cong H^2(\Gamma, \kk^\times)^F = 0$, as claimed.
\end{proof}

\begin{lemma}\label{comod-blmd} There is an equivalence of tensor categories
\begin{align*}
{B_\lambda}^*\comod &\cong \C(F\ltimes \Gamma, \omega, F, 1),
\end{align*}where the class of the 3-cocycle $\omega$ is trivial in $H^3(F\ltimes \Gamma, \kk^\times)$.
\end{lemma}

\begin{proof} Since the orders of $F$ and $\Gamma$ are relatively prime,  the class of the 3-cocycle $\omega$ in $H^3( F\ltimes \Gamma , \kk^{\times})$ associated to an exact sequence  $\kk \to \kk^\Gamma \to B_\lambda \to \kk F \to \kk$ is trivial, due to Lemma~\ref{lem-trivial}. This implies the lemma, in view of the equivalence~\eqref{gt-equiv}.
\end{proof}

\medbreak
Suppose now that $p \equiv1$ mod $q$. Fix  $h$ and $t$  primitive $q$th roots of $1$ modulo $p$. Consider the matched pair $(F', \Gamma')$, where
$$\Gamma' = \langle a, b: \, a^q = b^p = 1, aba^{-1} = b^t\rangle \cong \Z_p \rtimes \Z_q, \ \ F' = \langle g: \, g^q = 1 \rangle \cong \Z_q,$$
such that the action $\rhd: \Gamma' \times F' \to F'$ is trivial and the action $\lhd: \Gamma' \times F' \to \Gamma'$  is the action by group automorphisms determined by
\begin{equation}\label{action-pqq-2} a \lhd g = a, \quad b \lhd g = b^{h}.
\end{equation}
The associated group $F' \bowtie \Gamma'$ coincides with the semidirect product $F' \ltimes \Gamma'$ with respect to the action \eqref{action-pqq-2}.

\medbreak
For an integer $0 \leq l \leq q-1$,  Andruskiewitsch and  the third author described in \cite[Subsection 2.4]{examples}  a self-dual  semisimple Hopf algebra $A_l$ of dimension $pq^2$ that fits into an abelian exact sequence $\kk \to \kk^{\Gamma'} \to A_l\to \kk F' \to \kk$;  see also \cite[Section~1.3]{natale-pqq}.

\medbreak
Fix $y \in \kk^{\Gamma'}$ a group-like element of order $q$. That is, $y: \Gamma' \to \kk^\times$ is a non-trivial group homomorphism.

\begin{definition}[$A_l$] \label{A-l}
The Hopf algebra $A_l$ is defined as  the bicrossed product  $\kk^{\Gamma'}\!\#_{\sigma^{(l)}}
	\kk F'$  associated to the matched pair $(F', \Gamma')$,  where the cocycle \linebreak $\sigma^{(l)}:F' \times F'\to \kk^{\Gamma'}$ is given by
$$\sigma^{(l)}(g^n,g^{n'})= y^{l[\frac{n+n'}{q}]},$$
for all $0\leq n,m \leq q-1$.  \end{definition}

\begin{lemma}\label{comod-al} There is an equivalence of tensor categories  $$A_l\comod \cong \C(F'\ltimes \Gamma', \upsilon, F', 1)$$  where the 3-cocycle $\upsilon \in H^3(F'\ltimes \Gamma', \kk^{\times})$ is represented by
\begin{equation}\label{omega-al}\upsilon (g^nb^ja^i, g^{n'}b^{j'}a^{i'}, g^{n''}b^{j''}a^{i''})  = \eta^{li[\frac{n'+n''}{q}]}
\end{equation}
for all $0 \leq i, i', i'', n, n', n'' \leq q-1$, $0\leq j,j',j''\leq p-1$,  where $\eta \in \kk$ is the primitive $q$th root of unity defined by $\eta = y(a)$.
\end{lemma}

\begin{proof} Since $A_l$ is self-dual, we have that  $A_l\comod \cong \Rep A_l$.
Thus the result is a special case of the equivalence \eqref{gt-equiv}. Observe that $$\sigma^{(l)}_{b^{j}a^{i}}(g^{n'}, g^{n''}) = \sigma^{(l)}(g^{n'}, g^{n''}) (b^ja^i) = \eta^{li[\frac{n'+n''}{q}]},$$ for all $0 \leq i, i', i'', n, n', n'' \leq q-1$, $0\leq j,j',j''\leq p-1$.
\end{proof}

Finally, we have the classification result below.

\begin{theorem}\label{p11} \cite[Theorem 3.12.4]{natale-pqq} Let $H$ be a noncommutative, noncocommutative, semisimple Hopf algebra of dimension $pq^2$.  Then $H$ is isomorphic to precisely one of the Hopf algebras
$$A_0,\ldots, A_{q-1},
B_{\lambda_1},\ldots,B_{\lambda_n},B_{\lambda_1}^*,\ldots,B_{\lambda_n}^*,$$
such that   $n = 1$ if $p = 2$,  and $n=\frac{p+1}{2}$ if $p > 2$,  where $0\leq\lambda_j<p-1$ are such that \linebreak $\lambda_i\lambda_j\not \equiv 1 \hspace{-.05in} \mod p$ if $i\neq j$. \qed
\end{theorem}

%%%%%%%%%%%%%%%%%%%%%%%%%%
%%%%%%%%%%%%%%%%%%%%%%%%%%
%%%%%%%%%%%%%%%%%%%%%%%%%%
%%%%%%%%%%%%%%%%%%%%%%%%%%

\section{Galois objects for semisimple Hopf algebras of dimension $pq^2$} \label{sec-main2}

In this section, we compute the number of right Galois objects of the bicrossed products $B_\lambda, B_\lambda^*$ and $A_l$, introduced in Definitions  \ref{B-lambda} and \ref{A-l}.

\begin{theorem} \label{main-teo} Recall the definition of the noncommutative, noncocommutative, semisimple Hopf algebras of dimension $pq^2$ $B_\lambda$, $B_\lambda^*$, given in Definition~\ref{B-lambda}, and $A_l$ (self-dual) given in Definition~\ref{A-l}. The following statements hold:

\begin{enumerate}
\item[(1)] Any right Galois object for the Hopf algebras $B_\lambda$, $A_l$, $l \neq 0$,  is trivial.

\smallskip

\item[(2)] The Hopf algebra $A_0$ has exactly $q$ right Galois objects up to isomorphism.

\smallskip

\item[(3)] Each of the Hopf algebras $B_\lambda^*$  has exactly $r+1$ right Galois objects up to isomorphism, where $q=pr+1$.
\end{enumerate}
\end{theorem}

\begin{proof}
 \underline{Case ${B_\lambda}$}. The Hopf algebra $B_\lambda^*$ fits into a central exact sequence
\begin{equation}\label{dual-b}\kk \to \kk^{F} \to B_\lambda^* \to \kk\Gamma \to \kk.
\end{equation}
Let $(\Gamma, F)$ be the matched pair associated to \eqref{dual-b}, so that the action $F \times \Gamma \to F$ is trivial and $\Gamma \bowtie F$ is a semidirect product $\Gamma \rtimes F$. Observe that $\Gamma \rtimes F$ is isomorphic to the semidirect product $F\ltimes \Gamma$ associated to the exact sequence $$\kk \to \kk^\Gamma \to B_\lambda(m,\zeta) \to \kk F \to \kk,$$ see e.g. \cite[Exercise 5.5]{ma-ext}.

\medbreak From \eqref{gt-equiv}, we obtain an equivalence of tensor categories
\begin{equation*}B_\lambda\comod \cong \C(\Gamma \rtimes F, \omega, \Gamma, 1).
\end{equation*}
By Proposition~\ref{bijcorr}, right Galois objects of $B_\lambda$ correspond to fiber functors of the category $ \C(\Gamma\rtimes F, \omega, \Gamma, 1)$.
Moreover, by Lemma \ref{lem-trivial} the class of the 3-cocycle $\omega$ induced from \eqref{dual-b} is trivial. By Theorem \ref{th-ostrik}, these fiber functors arise from pairs $(L, \beta)$ where $L$ is a subgroup of $\Gamma\rtimes F$ and $\beta$ is  a 2-cocycle on $L$, such that the following conditions are satisfied:
\begin{enumerate}
\item[(i)] $\Gamma\rtimes F = L\Gamma$.
\item[(ii)] The class of the 2-cocycle $\beta\vert_{\Gamma\cap L}$  is non-degenerate.
\end{enumerate}
Suppose $(L, \beta)$ is such a pair.
Condition (ii) implies that $\Gamma \cap L = 1$ or  $\Gamma\subseteq L$.

\medbreak Suppose first that $\Gamma\subseteq L$. Then $L = \Gamma\rtimes F$, because of  condition (i). By Lemma~\ref{h2-g}, $H^2(F\ltimes \Gamma, \kk^\times) = 0$, so this possibility is discarded by condition~(ii).

\medbreak Suppose next that $\Gamma\cap L = 1$. Then condition (i) forces $|L|=p$.   This implies that $\beta=1$ since $H^2(L,\kk^\times)$ is trivial.
Moreover, by Theorem~\ref{th-ostrik}, isomorphism classes of fiber functors are classified by conjugacy classes of such pairs;  see Remark~\ref{rmk-conjclass}.
Now, $L$ is a $p$-Sylow subgroup of $\Gamma\rtimes F$ and hence $L$ must be a conjugate of $F$.
Therefore there is in this case a unique conjugacy class represented by the pair $(F, 1)$, which corresponds to the trivial $B_\lambda$-Galois object.

\medbreak
 \underline{Case $B_\lambda^*$}.  By Proposition~\ref{bijcorr} and Lemma~\ref{comod-blmd}, right Galois objects of $B_\lambda^*$ correspond to fiber functors of the category $\C(F\ltimes \Gamma, \omega, F, 1)$, where the class of the 3-cocycle $\omega$ is trivial. In particular, there is an equivalence of tensor categories $\C(F\ltimes \Gamma, \omega, F, 1) \cong \C(F\ltimes \Gamma, 1, F, \alpha)$, where $\alpha$ is a 2-cocycle on $F$, see for instance \cite[Remark~8.39]{ENO}.  In view of Theorem~\ref{th-ostrik} and Remark~\ref{rmk-conjclass}, these fiber functors are classified by conjugacy classes of pairs $(L, \beta)$ where $L$ is a subgroup of $F\ltimes \Gamma$ and $\beta$ is  a 2-cocycle on $L$, such that the following conditions are satisfied:

\begin{enumerate}
\item[(i)] $F\ltimes \Gamma = LF$.
\item[(ii)] The class of the 2-cocycle $\alpha^{-1}\vert_{F\cap L}\beta\vert_{F\cap L}$ is non-degenerate.
\end{enumerate}

\smallskip
Suppose $(L, \beta)$ is such a pair.
Condition (ii) implies that $F\cap L =1$. Then  $L$ must be a subgroup of order $q^2$ of $F\ltimes \Gamma$, in view of condition (i).
Then $L$ is a $q$-Sylow subgroup of $F\ltimes \Gamma$, and there is only one conjugacy class of such subgroups. So we can take $L=\Gamma$ and condition (i) is satisfied. Since $F\cap L$ is trivial, condition (ii) holds as well.

\medbreak
Now we need to determine the conjugacy classes of 2-cocycles on $\Gamma$.  We have that $H^2(\Gamma,\kk^\times)\simeq \Z_q$ and every 2-cocycle is cohomologous to exactly one of the form $$\sigma_\xi(a^ib^j,a^{i'}b^{j'})=\xi^{-i'j}, \quad  0 \leq i,i',j,j' \leq  q-1,$$
where $\xi$ is a primitive $q$th root of $1$. Then for $z= g^ka^ib^j\in F\ltimes \Gamma$, $0 \leq k \leq  p-1$, $0 \leq i,j \leq q-1$, we get from \eqref{action-pqq} that
\begin{align*}
\sigma_\xi\big{(}za^{i'}b^{j'}z^{-1},za^{i''}b^{j''}z^{-1}\big{)}&=\sigma_\xi(a^{i'm^k}b^{j'm^{k\lambda}},a^{i''m^k}b^{j''m^{k\lambda}})\\
&=\xi^{-j'm^{k\lambda}m^ki''},
\end{align*}
for all $0 \leq i', j', i'', j'' \leq q-1$.
Therefore $\sigma_\xi^z =  \sigma_{\xi^{m^{k(\lambda+1)}}}$. Thus, the conjugacy class of $\sigma_\xi$ is $$\{\sigma_{\xi^{m^{k(\lambda+1)}}} |  \ \ k=0,\ldots, p-1\}.$$
Since $\lambda \not \equiv -1~\text{mod } p$, and $m$ is a primitive $p$th root of unity modulo $q$, then $m^{k(\lambda+1)}\not \equiv 1$ mod $q$, for all $0 < k \leq p-1$. Therefore the conjugacy class corresponding to $\xi \neq 1$ has $p$ elements. Write $q=pr+1$. Then we obtain $r$ conjugacy classes with $p$ elements each and one class with 1 element. Denote the representative cocycles by $\beta_1,\ldots,\beta_r,1$.

\medbreak
Then, in this case, $(\Gamma,\beta_1),\ldots,(\Gamma,\beta_r),(\Gamma,1)$ represent the conjugacy classes of pairs $(L,\beta)$ giving rise to a fiber functor.

\medbreak

 \underline{Case $A_l$}.  In view of Proposition~\ref{bijcorr} and Lemma \ref{comod-al}, right Galois objects for the Hopf algebra $A_l$ correspond to fiber functors of the category $\C(F'\ltimes \Gamma', \upsilon, F', 1)$  where the 3-cocycle $\upsilon$ on $F' \ltimes \Gamma'$ is  given by Formula \eqref{omega-al}.

\medbreak
A fiber functor corresponds to a pair $(L, \beta)$ where $L$ is a subgroup of $F'\ltimes \Gamma'$ and $\beta$ is  a 2-cocycle on $L$, such that the following conditions are satisfied:
\begin{enumerate}\item[(i)] The class of $\upsilon\vert_L$ is trivial.
\item[(ii)] $F'\ltimes \Gamma' = LF'$.
\item[(iii)] The class of the 2-cocycle $\beta\vert_{F'\cap L}$ is non-degenerate.
\end{enumerate}

\medbreak
Suppose $(L, \beta)$ is such a pair. Conditions (ii) and (iii) imply that $L$ is a subgroup of order $pq$ of $F' \ltimes \Gamma'$ such that $F' \cap L = 1$.
Since $H^2(L,\kk^\times) = 0$, we have that $\beta=1$. Note in addition that every subgroup of order $pq$ is normal in $F' \ltimes \Gamma'$. It follows from Theorem \ref{th-ostrik} that isomorphism classes of fiber functors are classified in this case by pairs $(L, 1)$, where $L$ is a subgroup of order $pq$ satisfying condition~(i) such that $L\cap F' = 1$; see  Remark \ref{rmk-conjclass}.

\medbreak We claim that $L$ is one of the (pairwise distinct) subgroups $L_k = \langle b, g^ka \rangle$, for some $0 \leq k \leq q-1$.
Indeed, since $L$ is of order $pq$, then it contains the subgroup $\langle b\rangle$ generated by $b$ and therefore $L = \langle b\rangle S$, where $S$ is a subgroup of order $q$ such that  $S \neq F'$, by condition (i).
Each subgroup of order $q$ is inside of some $q$-Sylow subgroup, then it is conjugate to one inside $\langle g, a\rangle\subset F'\ltimes\Gamma'$; that is, $S$ is conjugate to one of the subgroups $S_k=\langle g^ka\rangle$, $0\leq k\leq q-1$.
Hence, there must exist $h \in F'\ltimes\Gamma'$ and $0 \leq k \leq q-1$ such that $hSh^{-1} = S_k$. Then $L = hLh^{-1} = (h \langle b\rangle h^{-1}) \; (hSh^{-1}) = L_k$. This proves the claim.

\medbreak
Now we study condition (i).   By Equation \eqref{omega-al}, it will be enough to analyze the restriction of $\upsilon$ to the cyclic subgroup $S_k = \langle g^ka\rangle$ of $L_k$.
For all $n,m,t=0,\ldots,q-1$, we have
\begin{align}\label{rest-ups}
\upsilon((g^ka)^n,(g^ka)^{n'},(g^ka)^{n''})&=
\eta^{ln[\frac{n'k\;(\text{mod} q) \; + \; n''k\;(\text{mod} q) }{q}]}.
\end{align}

If $l=0$, conditions (i) and (ii) hold for all $L_k$, $0 \leq k \leq q-1$. Then we obtain $q$ fiber functors $(L_0,1),\ldots,(L_{q-1},1)$ that  correspond to $q$ non-isomorphic  right Galois objects.

\medbreak
Suppose that $l\neq 0$. If $k= 0$, $L_0 = \Gamma'$, thus $(L_0,1)$ corresponds to the trivial right Galois object.

\medbreak Assume next that $k\neq 0$. Let $\tilde\eta \in \kk$ be a primitive $q$th root of unity such that ${\tilde \eta}^{k} = \eta^l$.
From Formula \eqref{rest-ups} we obtain
\begin{align*}
\upsilon((g^ka)^n,(g^ka)^{n'},(g^ka)^{n''}) & =
\tilde\eta^{kn[\frac{n'k\;(\text{mod} q) \; + \; n''k\;(\text{mod} q) }{q}]} \\
& = \upsilon_{\tilde{\eta}}((ga)^{kn},(ga)^{kn'},(ga)^{kn''}),
\end{align*}
where $\upsilon_{\tilde{\eta}}: S_1 \times S_1\times S_1 \to \kk^\times$ is the 3-cocycle given by $$\upsilon_{\tilde{\eta}}((ga)^{n},(ga)^{n'},(ga)^{n''}) = \tilde\eta^{n[\frac{n' +n''}{q}]},$$ for all $0\leq n, n', n'' \leq q-1$. In other words, $\upsilon\vert_{S_k\times S_k\times S_k}$ coincides with the image of $\upsilon_{\tilde{\eta}}$ under the group isomorphism $H^3(S_1, \kk^\times) \cong H^3(S_k, \kk^\times)$ induced by the isomorphism $f: S_k \to S_1$, defined by $f(g^ka) = (ga)^k$.
Since $\tilde \eta \neq 1$, then the class of  $\upsilon_{\tilde{\eta}}$ is not trivial in $H^3(S_1, \kk^\times)$ (see the discussion at the beginning of Section \ref{sec-main1}). Therefore $\upsilon\vert_{S_k\times S_k\times S_k}$ represents a non-trivial cohomology class. Thus,\linebreak condition~(i) is violated for pairs $(L_k,1)$ such that $k\neq 0$. Hence,  in this case there is, up to isomorphism, a unique fiber functor corresponding to the pair $(\Gamma',1)$.
\end{proof}

%%%%%%%%%%%%%%%%%%%%%%%%%%
%%%%%%%%%%%%%%%%%%%%%%%%%%
%%%%%%%%%%%%%%%%%%%%%%%%%%
%%%%%%%%%%%%%%%%%%%%%%%%%%

\section{Cocycle deformations of semisimple Hopf algebras \\of dimension $p^3$, $pq^2$} \label{s-cocycle}

In this section we discuss cocycle deformations for noncommutative, noncocommutative, semisimple Hopf algebras of dimension $p^3$ and $pq^2$.
We begin with the following observation:

\begin{remark} \label{Galois-cocycle}
If a finite-dimensional Hopf algebra $H$ only has trivial right Galois objects, then $H$ cannot be deformed by a cocycle non-trivially.  Indeed, the left Galois Hopf algebra of the trivial Galois object $H$ is isomorphic to $H$; see Proposition~\ref{prop-connect}.
\end{remark}

Recall the classification of the  noncommutative noncocommutative semisimple Hopf algebras of dimension $p^3$ in Theorem~\ref{ppp}. We have the following result for such Hopf algebras.

\begin{proposition}\label{cocycle-ppp} Each of the noncommutative, noncocommutative, semisimple Hopf algebras of dimension $p^3$  has only trivial cocycle deformations.
\end{proposition}

\begin{proof} It is known  that these (self-dual) Hopf algebras are not cocycle deformations of a dual group algebra by \cite[Proposition 2.5]{ma-contemp}. Moreover, Theorem \ref{cls-morita} implies that they are not cocycle deformations of each other; see Proposition~\ref{prop-connect} (5). This implies the proposition.
\end{proof}

In the next result, we discuss cocycle deformations of semisimple Hopf algebras of dimension $pq^2$; recall the classification of such Hopf algebras in Theorem~\ref{p11}.

\begin{proposition} \label{cocycle-pqq}
 For the cocycle deformations of the noncommutative, noncocommutative, semisimple Hopf algebras of dimension $pq^2$, the following statements hold.

\begin{enumerate}\item[(a)] The Hopf algebras  $A_1,\dots,A_{q-1}$ and $B_{\lambda_1},\ldots,B_{\lambda_{n}}$
have only trivial cocycle deformations.

\smallskip

\item[(b)] Each of the Hopf algebras $A_0$ and  $B_{\lambda_1}^*,\ldots,B_{\lambda_{n}}^*$ is a cocycle deformation of a dual group algebra.
\end{enumerate}
\end{proposition}

\begin{proof}
Part (a) follows from Theorem~\ref{main-teo}; see Remark~\ref{Galois-cocycle}. Part (b) follows from \cite[Propositions~5.2.1(i) and~5.3.1(i)]{N}, along with the self-duality of $A_0$. \end{proof}

%%%%%%%%%%%%%%%%%%%%%%%%%%
%%%%%%%%%%%%%%%%%%%%%%%%%%
%%%%%%%%%%%%%%%%%%%%%%%%%%
%%%%%%%%%%%%%%%%%%%%%%%%%%

\appendix

\section{Cocycle deformations and Galois objects for the Kac-Paljutkin Hopf algebra}\label{kp}

 In this appendix we recover Masuoka's result that the noncommutative, noncocommutative, semisimple Kac-Paljutkin Hopf algebra $H_8$ of dimension~8 has no non-trivial Galois objects, and no non-trivial cocycle deformations  \cite[Theorems~4.1(1) and~4.8(1)]{ma-contemp}.   This is achieved by using the techniques above.

\medbreak  According to \cite[Proposition~2.3]{mas-6-8}, $H_8 \cong \kk^F {}^\tau \#_{\sigma} \kk\Gamma$ is a bicrossed product where
$$F = \langle a,b : a^2 = b^2 = 1, ab=ba \rangle \cong \mathbb{Z}_2 \times \mathbb{Z}_2 \quad \text{and} \quad \Gamma = \langle t : t^2 = 1 \rangle \cong \mathbb{Z}_2$$
 is a matched pair of finite groups,
such that the action $\triangleleft : \Gamma \times F \to \Gamma$ is trivial and the action $\triangleright : \Gamma \times F \to F$ is the action by group automorphisms determined by
$$t \triangleright a = b, \quad \quad t \triangleright b = a.$$
The Hopf algebra $H_8$ is also self-dual. Here, the group $F \rtimes \Gamma$ is isomorphic to the dihedral group of order 8.

\begin{proposition} \label{prop-H8} Any right Galois object for $H_8$ is trivial, and hence, $H_8$ has no non-trivial cocycle deformations.
\end{proposition}

\begin{proof}
By Proposition~\ref{bijcorr}, right Galois objects of $H_8$ correspond to fiber functors of the category
$$H_8\comod ~~\cong  ~~ \C(F \rtimes \Gamma, \omega, F, 1),$$  where $\omega$ is the 3-cocycle on $F \rtimes \Gamma$ associated to $H_8$ in the Kac exact sequence.   In view of Theorem \ref{th-ostrik}, every such fiber functor corresponds to a pair $(L, \beta)$ where $L$ is a subgroup of $G:=F \rtimes \Gamma$ and $\beta$ is  a 2-cocycle on $L$, such that the following conditions are satisfied:

\smallskip

\begin{enumerate}\item[(i)] The class of $\omega\vert_{L\times L \times L}$ is trivial.
\item[(ii)] $G = LF$.
\item[(iii)] The class of the 2-cocycle $\beta\vert_{F\cap L}$ is non-degenerate.
\end{enumerate}

\smallskip

Suppose $(L, \beta)$ is such a pair.
As in the proof of Theorem~\ref{mainGalois}, the conditions above imply that $L$ is a subgroup of order $2$ of
$$F \rtimes \Gamma   = \langle a,b,t : a^2 = b^2 = t^2 =1, ~ab=ba, ~ta=bt, ~tb=at\rangle,$$ which is not contained in $F$. Thus, by Remark \ref{rmk-conjclass}, isomorphism classes of fiber functors on $H_8\comod$ are classified by conjugacy classes of such pairs $(L, \beta)$.
It is easy to see that the subgroups of order 2 not contained in $F$ are
$\Gamma= \langle t \rangle$ and $\langle abt \rangle$.
Moreover, these subgroups are conjugate as
$a (abt) a^{-1}= a (abt) a= t$.
Therefore, the only subgroup $L$ under consideration is $\Gamma \cong \mathbb{Z}_2$, and $\beta$ is trivial in this case. As in the proof of Theorem~\ref{mainGalois}, the pair $(\Gamma, 1)$ corresponds to the trivial Galois object.  This implies the result; see Remark ~\ref{Galois-cocycle}.
\end{proof}

%%%%%%%%%%%%%%%%%%%
%%%%%%%%%%%%%%%%%%%
%%%%%%%%%%%%%%%%%%%
%%%%%%%%%%%%%%%%%%%

\section{Categorical Morita equivalence classes among semisimple Hopf algebras of dimension $p^3$ (by S. Natale)}\label{morita-ppp}
	
Let $p$ be an odd prime number. In this appendix we determine the categorical Morita equivalence classes among the fusion categories of finite-dimensional representations of semisimple Hopf algebras of dimension $p^3$.

\medbreak Let $G = \UT(3, p)$ be the group of upper triangular unipotent $3\times 3$ matrices with entries in $\F_p$. We shall indicate by $T = \Z_{p^2} \rtimes \Z_p$  the unique nonabelian group of order $p^3$ and exponent $p^2$. The group $T$ has the following presentation by generators and relations:
\begin{equation*}T = \langle y, z:\; y^{p^2} = z^p = 1, \; [y, z] = y^p\rangle.
\end{equation*}

Let us recall the classification of semisimple Hopf algebras of dimension $p^3$, obtained by Masuoka in 1995:

\begin{theorem}{\cite[Theorem 3.1]{masuoka-pp}.}\label{cls-ppp} A semisimple Hopf algebra $A$ of dimension $p^3$ over $\kk$ is isomorphic to precisely one of the $p+8$ Hopf algebras listed below:
	
\medbreak (1) $\kk^{\Z_p \times \Z_p \times \Z_p}$, $\kk^{\Z_p \times \Z_{p^2}}$, $\kk^{\Z_{p^3}}$, $\kk^G$, $\kk^T$, $\kk G$, $\kk T$;
	
\medbreak (2) $A_{\zeta, 1}$, $A_{\zeta^t, 1}$, $A_{\zeta, g}$, $\dots$, $A_{\zeta^{p-1}, g}$, where $1\neq \zeta \in \kk$ is a $p$th root of 1, $1\neq g \in \kk^\Gamma$ is an $F$-invariant group-like element, and $t \in \mathbb F_p$ is a quadratic nonresidue.
\end{theorem}
See Definition \ref{def-A} for a presentation of the Hopf algebras in Theorem \ref{cls-ppp} (2).

\medbreak
Recall that two fusion categories  $\C$ and $\D$ are called categorically Morita equivalent  if $\D^{op}$ is equivalent to the category $\C_{\M}^*$, for some indecomposable $\C$-module category $\M$, and two finite-dimensional Hopf algebras $H$ and $L$ are called categorically Morita equivalent if $H\comod$ and $L\comod$ are categorically Morita equivalent.

\medbreak
The main result of this appendix is the following theorem, that classifies the categorical Morita equivalence classes among the Hopf algebras in Theorem \ref{cls-ppp}.

\begin{theorem}\label{cls-morita} Semisimple Hopf algebras of dimension $p^3$ fall into $p+6$ categorical Morita equivalence classes. More precisely, the Hopf algebras
\begin{equation}\label{list-clsmorita} \kk^{\Z_p \times \Z_p \times \Z_p}, \; \kk^{\Z_p \times \Z_{p^2}}, \; \kk^{\Z_{p^3}},  \kk^G, \; \kk^T, \; A_{\zeta, 1}, \; A_{\zeta^t, 1}, \; A_{\zeta, g}, \; \dots, \; A_{\zeta^{p-1}, g},
\end{equation} are pairwise categorically Morita inequivalent and, furthermore,
the equivalence class of $\kk^G$ (respectively, the equivalence class of $\kk^T$) consists of $\kk^G$ and $\kk G$ (respectively, of $\kk^T$ and $\kk T$).
\end{theorem}

The proof of Theorem \ref{cls-morita} will be given in Subsection \ref{pf}.

\subsection{Preliminary lemmas}
We begin by recording the following consequence of \cite[Theorem 3.1]{ENO2}.

\begin{lemma}\label{cmmttv} Let $L$ be a finite abelian group. Suppose that $\C$ is a fusion category such that $\C$ is categorically Morita equivalent to  $\C(L, 1)$. Then $\C$ is pointed.

Therefore, if $\C = \Rep H$, where $H$ is a finite-dimensional Hopf algebra, then $H$ is commutative.
\end{lemma}

\begin{proof} Let $\mathcal Z(\C)$ be the Drinfeld center of $\C$. Notice that
$\mathcal Z(\C(L, 1)) \cong D(L)\lRep$, as braided fusion categories, where $D(L)$ is the Drinfeld double of the group algebra $\kk L$. Since $L$ is abelian, then $D(L)$ is a commutative Hopf algebra and thus $\mathcal Z(\C(L, 1))$ is a pointed fusion category.
In view of \cite[Theorem 3.1]{ENO2}, the assumption implies that $\mathcal Z(\C(L, 1))$ and $\mathcal Z(\C)$ are equivalent as braided fusion categories.
Thus $\mathcal Z(\C)$ is a pointed fusion category and therefore so is $\C$.
\end{proof}

We next recall a Morita invariant of a fusion category introduced by Etingof in  \cite{etingof}.
Let $\C$ be a finite tensor category over $\kk$ and let
$\mathcal Z(\C)$ be its Drinfeld's center
The \emph{quasi-exponent} $\operatorname{qexp} \C$ of $\C$ is defined as the
smallest integer $N$ such that $(c^2)^N$ is unipotent, where $c^2$ is the 'square' of the canonical braiding $c$ of $\mathcal Z(\C)$, that is, $c^2_{X, Y} = c_{Y,
X}c_{X, Y}: X \otimes Y \to X \otimes Y$, $X, Y \in \mathcal Z(\C)$.

If $\C$ is a \emph{fusion} category, then the quasi-exponent of $\C$ is called the \emph{exponent} of $\C$ and
denoted $\exp \C$. By \cite[Theorem 5.1]{etingof}, the exponent of a fusion category is finite.
Note in addition that if $L$ is any finite group, the exponent of the pointed fusion category $\C(L, 1)$ coincides with the exponent of $L$.

Moreover, the exponent of a fusion category $\C$ is a Morita invariant of $\C$:

\begin{lemma}\label{exp-pt} Let $\C$, $\D$ be fusion categories and suppose that $\C$ and $\D$ are categorically Morita equivalent. Then $\exp \C = \exp \D$.
\end{lemma}

\begin{proof} By \cite[Proposition 6.3]{etingof}, we have
$\exp \C = \exp \mathcal Z(\C)$, for any fusion category $\C$. This implies the lemma since, by \cite[Theorem 3.1]{ENO2}, $\mathcal Z(\C)$ and $\mathcal Z(\D)$ are equivalent as braided fusion categories. \end{proof}

\subsection{The noncommutative examples}
Let $p$ be an odd prime number. Keep the notation in Section \ref{s-ppp}. Recall
that $\Gamma = \langle x: \, x^p = 1 \rangle \cong \Z_p,$ and $F = \langle a, b: \, a^p = b^p = 1, ab = ba\rangle \cong \Z_p \times \Z_p$. In addition, $(F, \Gamma)$ is a matched pair with respect to the action \eqref{action-ppp} such that $F \bowtie \Gamma = F \rtimes \Gamma \cong G$.

\medbreak Let $A_{\zeta, g}$ be one of the (noncommutative) Hopf algebra of dimension $p^3$ constructed by Masuoka in \cite[Lemma 2.4]{masuoka-pp}, where $\zeta \in \kk$ is a $p$th root of unity, $g$ is a $\Gamma$-invariant group-like element of $\kk^F$,  and $t \in \F_p$ is a quadratic nonresidue; see Definition \ref{def-A}.

As was observed in \cite[Example 2.6]{masuoka-pp}, there are isomorphisms of Hopf algebras $A_{1, 1} \cong \kk G$ and  $A_{1, g} \cong \kk T$, for all $g \neq 1$.

\begin{lemma}\label{equivalencia} Let $\zeta, \lambda \in \kk$ be $p$th roots of unity and let $g \in \kk^F$ be the group-like element defined by $g(a^ib^j) = \lambda^j$. Let also $\omega_{\zeta, \lambda}: G \times G \times G \to \kk$ be the 3-cocycle given by Formula \eqref{omega}.
Then the categories $\Rep A_{\zeta, g}$ and $\C(G, \omega_{\zeta, \lambda})$ are categorically Morita equivalent.
\end{lemma}

\begin{proof} Observe that the category $\C(G, \omega_{\zeta, \lambda}, F, 1)$ is categorically Morita equivalent to the pointed fusion category $\C(G, \omega_{\zeta, \lambda})$, for all $p$th roots unity $\zeta, \lambda \in \kk$.
By Proposition \ref{iso}, there is an isomorphism of Hopf algebras $A_{\zeta, g}^* \cong H_{\zeta, \lambda}$.
Hence, from Proposition  \ref{cat-eq}, we get
equivalences of tensor categories
\begin{equation}\Rep A_{\zeta, g}^* \cong \Rep H_{\zeta,\lambda}  \cong  \C(G, \omega_{\zeta, \lambda}, F, 1). \end{equation}	
This implies the lemma because the categories $\Rep A_{\zeta, g}^*$ and $\Rep A_{\zeta, g}$ are categorically Morita equivalent \cite[Theorem 4.2]{ostrik}.
\end{proof}

\begin{corollary}\label{class-omega} Suppose that $1\neq \zeta \in \kk$.
Then the class of the 3-cocycle $\omega_{\zeta, \lambda}$ is not trivial in $H^3(G, \kk^{\times})$.
\end{corollary}
	
\begin{proof}  Let $g \in \kk^F$ be the invariant group-like element such that $g(a^ib^j) = \lambda^j$, $0 \leq i, j \leq p-1$.  By Proposition \ref{iso}, there is an isomorphism of Hopf algebras $H_{\zeta, \lambda} = \kk^\Gamma {}^\tau\!\#_{\sigma}\kk F \cong A_{\zeta, g}^* \cong A_{\zeta, g}$.
Recall that the class of $\omega_{\zeta, \lambda}$ is the image of the abelian extension $\kk^\Gamma {}^\tau\!\#_{\sigma}\kk F$ under the map $\Opext(\kk^\Gamma, \kk F) \to H^3(G, \kk^\times)$ in the Kac exact sequence; see Subsection \ref{extn}.

\medbreak
Suppose on the contrary that the class of $\omega_{\zeta, \lambda}$ is trivial.
By exactness of the Kac sequence, the extension $\kk^\Gamma {}^\tau\!\#_{\sigma}\kk F$ belongs to the image of the map $$\delta: H^2(F, \kk^\times) = H^2(F, \kk^\times) \oplus H^2(\Gamma, \kk^\times) \to \Opext(\kk^\Gamma, \kk F),$$ and therefore its class in the cokernel of $\delta$ is trivial.
It follows from \cite[Proposition~3.1]{ma-newdir} that $H_{\zeta, \lambda} = \kk^\Gamma {}^\tau\!\#_{\sigma}\kk F$ is isomorphic to a cocycle deformation of the split extension $\kk^\Gamma \# \kk F$. But $\kk^\Gamma \# \kk F \cong \kk^G$ is a commutative Hopf algebra.
This leads to a contradiction because, for all $\zeta \neq 1$, the Hopf algebras $H_{\zeta, g} \cong A_{\zeta, g}$ are not a cocycle deformation of any commutative Hopf algebra \cite[Proposition~2.5]{ma-newdir}. Therefore the class of $\omega_{\zeta, \lambda}$ is not trivial, as claimed.
\end{proof}
	
Identify $G$ with the group $F \rtimes \Gamma$, and let $\omega = \omega_{\zeta, \lambda}$. For every pair $(L, \beta)$, where $L$ is a subgroup of $G$ such that the class of $\omega|_{L \times L \times L}$ is trivial and $\beta \in C^2(L, \kk^\times)$ is a 2-cocycle, let
$\M_0(L, \beta) = \C(G, \omega)_{\kk_\beta L}$
denote the indecomposable module category over $\C(G, \omega)$ corresponding to the  pair  $(L, \beta)$.  Recall that the category $(\C(G, \omega)^*_{\M_0(L, \beta)})^{op}$
is equivalent to the fusion category $\C(G, \omega, L, \beta) = {}_{\kk_\beta L}\C(G, \omega)_{\kk_\beta L}$ of $\kk_\beta L$-bimodules in $\C(G, \omega)$. Recall in addition that an indecomposable module category $\M_0(L, \beta)$ is  \emph{pointed} if the category $\C(G, \omega, L, \beta)$ is pointed.
	
\medbreak  Note that, since the restriction of $\omega_{\zeta, \lambda}$ to the cyclic subgroup $Z : = \langle a \rangle \cong \mathbb Z_p$ is trivial, then the pair $(Z, 1)$ has an associated indecomposable module category $\M_0(Z, 1)$.
	
\begin{proposition}\label{pointed-centro} Let $\zeta, \lambda \in \kk$ be $p$th roots of unity. Then $\M_0(Z, 1)$ is a pointed module category over $\C(G, \omega_{\zeta, \lambda})$. Furthermore, there is an equivalence of tensor categories $$\C(G, \omega_{\zeta, \lambda}, Z, 1) \cong \C(\tilde G, \tilde \omega_{\zeta, \lambda}),$$
where  $\tilde G$ is an  abelian group and $\tilde \omega_{\zeta, \lambda}$ is a 3-cocycle on $\tilde G$.
\end{proposition}
	
\begin{proof} The subgroup $Z$ coincides with the center of $G$. Therefore, it follows from \cite[Theorem 3.4]{naidu} that $\M_0(Z, 1)$ is a pointed module category. Hence $$\C(G, \omega_{\zeta, \lambda}, Z, 1) \cong \C(\tilde G, \tilde\omega_{\zeta, \lambda}),$$ where the group $\tilde G$ is isomorphic to the group  of invertible objects of the category $\C(G, \omega_{\zeta, \lambda}, Z, 1)$, and $\tilde \omega_{\zeta, \lambda}$ is a 3-cocycle on $\tilde G$.
		
\medbreak
We show next that the group $\tilde G$ is abelian. By \cite[Theorem 5.2]{gel-na}, the group $\tilde G$ is isomorphic to the group $K \ltimes_{\nu} \widehat{Z}$ defined as follows:
		
\medbreak
Take $\widehat{Z}:=\text{Hom}(Z, \kk^\times)$ and let $K = G/Z$. Then $\tilde G = K \times \widehat{Z}$ with the multiplication given by
\begin{equation}\label{g-bar} (t_1, \rho) . (t_2, \rho') = (t_1t_2, \nu(t_1, t_2) \, \rho \, ({}^{t_1}\!\rho')), \quad t_1, t_2 \in K, \ \  \rho, \rho' \in \widehat{Z},
\end{equation}
where the action $K \times \widehat{Z} \to \widehat{Z}$, $t_1 \times \rho \mapsto {}^{t_1}\!\rho$, is induced from the adjoint action of $K$ on $Z$, and  $\nu: K \times K \to \widehat{Z}$ is a certain cocycle on $K$. Notice that since $Z$ is central in $G$, then the action of $K$ on $\widehat{Z}$ is trivial.
		
\medbreak
Let $\omega = \omega_{\zeta, \lambda}$. In our case, using that $Z$ is central in $G$, that the cocycle $\omega\vert_{Z \times Z \times Z}$ is trivial, and that the class of the 2-cocycle involved in the definition of $\M_0(Z, 1)$ is trivial as well, the formula given for the cocycle $\nu$ in \cite[Formula (7)]{gel-na} reduces to the following expression:
\begin{equation*}\nu(t_1, t_2)(h) = \frac{\eta_{t_1}\!(h) \, \eta_{t_2}(h)}{\eta_{t_1t_2}(h)} \; \frac{\omega(g_1, h, h^{-1}) \, \omega(g_2, h, h^{-1}) \, \omega(h, g_2, \kappa(t_1, t_2))}
{\omega(g_2, h, \kappa(t_1, t_2)) \, \omega(g_2, h\kappa(t_1, t_2), h^{-1})},
\end{equation*}
for all $t_1, t_2 \in K$, $h \in Z$. Here, $g_i \in G$ is a representative of the class $t_i \in G/Z$,  $\kappa(t_1, t_2) \in Z$ is defined by $g_1g_2 = g_3\kappa(t_1, t_2)$, where $g_3 \in G$ is a representative of the class $t_1t_2$.  Also, for each $t \in K$, $\eta_t: Z \to \kk^\times$ is a 1-cochain such that $d\eta_t = \psi^t$,  where  $\psi^t$ is a certain 2-cocycle on $Z$. See \cite[Subsection 5.2]{gel-na}.
		
\medbreak We may write $K = \langle s, t:\; s^p = t^p = sts^{-1}t^{-1} = 1\rangle$, where $s = \pi(b)$, $t = \pi(x)$, and $\pi:G \to G/Z$ is the canonical surjection.
Thus, for all $0 \leq i, j \leq p-1$,
$$\kappa(s^it^j, s^{i'}t^{j'}) = a^{-ji'}.$$
It follows from Formula \eqref{omega} for $\omega$ that
\begin{equation*}\frac{\omega(b^ix^j, a^l, a^{-l}) \, \omega(b^{i'}x^{j'}, a^l, a^{-l}) \, \omega(a^l, b^{j'}x^{i'},  a^{-ji'}) }
{\omega(b^{i'}x^{j'}, a^l, a^{-ji'}) \, \omega(b^{i'}x^{j'}, a^{l-ji'}, a^{-l})} = 1,
\end{equation*}
for all $0\leq i, i', j, j', l \leq p-1$.
The cocycle $\nu$ becomes, for all $0\leq i, i', j, j', l \leq p-1$,
\begin{equation*}
\nu(s^it^j\! , s^{i'}t^{j'}) (a^l) = \frac{\eta_{s^it^j}(a^l) \, \eta_{s^{i'}t^{j'}}(a^l)}{\eta_{s^{i+i'}t^{j+j'}}(a^l)}.
\end{equation*}
Hence we obtain that $\nu(t_1, t_2) = \nu(t_2, t_1)$, for all $t_1, t_2 \in G/Z$.
Therefore, the product \eqref{g-bar} is commutative, and $\tilde G$ is abelian.   This finishes the proof of the proposition.
\end{proof}

\begin{remark} An explicit expression for the 3-cocycle $\tilde \omega_{\zeta, \lambda}$ in Proposition \ref{pointed-centro} can be obtained from \cite[Formula (24)]{naidu}. However, this expression will not be needed in what follows.
\end{remark}
	
\begin{proposition}\label{main-morita} Let $\zeta, \lambda \in \kk$ be fixed primitive $p$th roots of unity, and let $t \in \F_p$ be a quadratic nonresidue. Then the pointed fusion categories
\begin{equation}\label{zeq1}\C(G, \omega_{1, 1}), \;
\C(G, \omega_{1, \lambda}), \end{equation}
\begin{equation}\label{zneq1}\C(G, \omega_{\zeta, 1}), \; \C(G, \omega_{\zeta^t, 1}), \; \C(G, \omega_{\zeta, \lambda}), \; \dots, \; \C(G, \omega_{\zeta^{p-1}, \lambda}),\end{equation}
are pairwise categorically Morita inequivalent.	
\end{proposition}

\begin{proof}
Let $\C$ be one of the categories in \eqref{zeq1}, \eqref{zneq1}, and let us denote by $\omega$ the 3-cocycle corresponding to $\C$, so that $\C = \C(G, \omega)$. Suppose that $\M_0(L, \beta)$ is a pointed fusion category over $\C$.
In particular, due to \cite[Theorem~3.4]{naidu}, $L$ is a normal abelian subgroup of $G$ such that the class of $\omega$ is trivial on $L$, and $\beta$ is a 2-cocycle on $L$. When $L = 1$, the category $(\C^*_{\M_0(1, 1)})^{op}$ is equivalent to $\C$.

\medbreak
The only normal subgroup of order $p$ of $G$ is its center $Z = \langle a \rangle$. Since the group $H^2(Z, \kk^\times)$ is trivial, then the only possible choice for $\beta$ is $\beta = 1$. By Proposition \ref{pointed-centro}, the category $(\C^*_{\M_0(Z, 1)})^{op} \cong \C(G, \omega_{\zeta, \lambda}, Z, 1)$
has an abelian group of invertible objects. Therefore  $(\C^*_{\M_0(Z, 1)})^{op}$
is not equivalent to any of the categories \eqref{zeq1}, \eqref{zneq1}.

\medbreak
The remaining normal abelian subgroups of $G$ are the subgroups $N_j = \langle a, b^jx\rangle$, where $0 \leq j \leq p-1$,  and $F = \langle a, b \rangle$.
Observe that each of these subgroups has an exact complement in $G$; in fact,  $G = \langle b \rangle N_j$,  and $\langle b \rangle \cap N_j = 1$, for all $0\leq j \leq p-1$, and $G = \Gamma F$, with $\Gamma \cap F = 1$.
In addition, the restrictions $\omega\vert_{\langle b \rangle \times \langle b \rangle\times \langle b \rangle}$ and $\omega\vert_{\Gamma \times \Gamma \times \Gamma}$ are both trivial; see Formula \eqref{omega}.

Denote $F$ by $N_p$ and suppose that the class of $\omega\vert_{N_j \times N_j \times N_j}$ is trivial and $\beta$ is a 2-cocycle on $N_j$, $0 \leq j \leq p$. In view of \cite{ostrik-dd}, the fusion category $(\C^*_{\M_0(N_j, \beta)})^{op} \cong \C(G, \omega, N_j, \beta)$ has a fiber functor.

\medbreak
It follows from Corollary \ref{class-omega} that the categories \eqref{zneq1}
do not admit a fiber functor. The previous discussion implies that if $\C$ is one of the categories in \eqref{zeq1}, then $\C$ is not Morita equivalent to any of the categories \eqref{zneq1}. Similarly, we get that the categories \eqref{zneq1} are pairwise categorically Morita inequivalent.

\medbreak
Finally, we know from Lemma \ref{equivalencia} that $\C(G, \omega_{1, 1})$ is categorically Morita equivalent to $\Rep \kk G$ and $\C(G, \omega_{1, \lambda})$ is categorically Morita equivalent to $\Rep \kk T$.
Since the exponent of $G$ is $p$ and the exponent of $T$ is $p^2$, Lemma \ref{exp-pt} implies that $\Rep \kk G$ and $\Rep \kk T$ are not categorically Morita equivalent. Hence, the categories \eqref{zeq1} are not categorically Morita equivalent neither.
This finishes the proof of the proposition.
\end{proof}

\subsection{Proof of main result}\label{pf}
In this subsection we combine the previous results to give a proof of Theorem \ref{cls-morita} and some of its consequences.

\begin{proof}[Proof of Theorem \ref{cls-morita}] It follows from \cite[Theorem 4.2]{ostrik}, that the categories $\Rep H$ and $\Rep H^*$ are categorically Morita equivalent, for any finite-dimensional Hopf algebra $H$. In particular, the categories $\C(L, 1) \cong \Rep \kk^L$ and $\Rep \kk L$ are categorically Morita equivalent, for  any finite group $L$.

\medbreak Therefore, in view of Theorem \ref{cls-ppp}, it will be enough to determine the categorical Morita equivalence classes among the categories
\begin{equation}\label{abelian}\Rep\kk^{\Z_p \times \Z_p \times \Z_p}, \;
\Rep\kk^{\Z_p \times \Z_{p^2}}, \; \Rep\kk^{\Z_{p^3}}, \text{ and}
\end{equation}
\begin{equation}\label{noncomm}\Rep \kk G, \; \Rep \kk T, \;
\Rep A_{\zeta, 1}, \; \Rep A_{\zeta^t, 1}, \; \Rep A_{\zeta, g}, \; \dots, \; \Rep A_{\zeta^{p-1}, g}, \end{equation}
where $\zeta \in \kk$ is fixed primitive $p$th root of unity, $1\neq g \in \kk^F$ is a $\Gamma$-invariant group-like element,  and $t \in \F_p$ is a quadratic nonresidue.

In view of Lemma \ref{equivalencia}
and Proposition \ref{main-morita}, the  categories \eqref{noncomm} are pairwise categorically Morita inequivalent.
In addition, it follows from Lemma \ref{exp-pt} that the fusion categories \eqref{abelian} are pairwise categorically Morita inequivalent as well.
Moreover, Lemma \ref{cmmttv} implies that none of the categories \eqref{abelian} is categorically Morita equivalent to any of the categories  \eqref{noncomm}.
This finishes the proof the theorem. \end{proof}

We conclude by stating some corollaries of Theorem \ref{cls-morita}.

\begin{corollary} Let $H$ and $L$ be semisimple Hopf algebras of dimension $p^3$. Then $H$ and $L$ are categorically Morita equivalent if and only if $H \cong L$ or $H \cong L^*$. \qed
\end{corollary}

For a 3-cocycle $\omega$ on a group $L$, let $D^\omega(L)$ denote the \emph{twisted quantum double} due to Dijkgraaf, Pasquier and Roche \cite{dpr}.
	
\begin{corollary}\label{dpr-ppp} Let $\zeta, \lambda \in \kk$ be fixed primitive $p$th roots of unity and let $t \in \F_p$ be a quadratic nonresidue. Then the categories of finite-dimensional representations of the (twisted) quantum doubles
$$D(\Z_p \times \Z_p \times \Z_p), \; D(\Z_p \times \Z_{p^2}), \; D(\Z_{p^3}),$$
$$D(T), \; D(G), \; D^{\omega_{\zeta, 1}}(G), \; D^{\omega_{\zeta^t, 1}}(G), \; D^{\omega_{\zeta, \lambda}}(G), \; \dots, \; D^{\omega_{\zeta^{p-1}, \lambda}}(G),$$
are pairwise inequivalent as braided fusion categories.
\end{corollary}

\begin{proof} Combining Theorem \ref{cls-morita} with \cite[Theorem 3.1]{ENO2}, we obtain that the Drinfeld centers of the fusion categories of finite-dimensional representations of the semisimple Hopf algebras of dimension $p^3$ are pairwise inequivalent as braided fusion categories.	The corollary follows from \cite[Theorem 1.3 (ii)]{gp-ttic}.
\end{proof}

%%%%%%%%%%%%%%%%%%%%%%%%%%

\section*{Acknowledgments}
We thank the referee for their comments that improved the exposition of this manuscript and for their comments toward Remark~\ref{referee}. This project began at the Women in Noncommutative Algebra and
Representation Theory (WINART) workshop at the Banff International
Research Station (BIRS) in March 2016.  We thank the BIRS
administration and staff for their hospitality and productive
atmosphere. S. Natale thanks the Mathematics Department of Temple University, Philadelphia, for the kind hospitality during her visit in March 2016.

\smallskip

The work of A. Mej\'{i}a  is partially supported by Capes-PNPD, CONICET and SECyT-UNC.  The research of S.
Natale is partially supported by CONICET and SECyT-UNC. S. Montgomery
is supported by the U.S. National Science Foundation (NSF), grant \#DMS-1301860.  C. Walton is also supported by the U.S. NSF, grant \#DMS-1550306; S. Natale's
visit to Temple University in Spring 2016 was funded by this grant.

\bibliographystyle{amsalpha}

\end{document}